\documentclass[11pt,a4]{article}
\usepackage{latexsym}
\usepackage{amsmath}
\usepackage{amssymb}
\usepackage{amsthm}
\usepackage{amscd}
\usepackage{color}
\usepackage{authblk}
\definecolor{dblue}{rgb}{0,0,0.45}
\definecolor{red}{rgb}{0.7,0,0}

\newtheorem{theorem}{Theorem}[section]
\newtheorem{lemma}[theorem]{Lemma}

\newtheorem{proposition}[theorem]{Proposition}

\theoremstyle{definition}

\newtheorem{definition}[theorem]{Definition}

\theoremstyle{remark}

\newcommand{\cM}{{\mathcal M}}
\newcommand{\cMpq}{{\mathcal M}^p_q}
\newcommand{\N}{{\mathbb N}}
\newcommand{\R}{{\mathbb R}}
\newcommand{\C}{{\mathbb C}}

\newcommand{\Z}{{\mathbb Z}}

\newcommand{\cF}{{\mathcal F}}
\newcommand{\cG}{{\mathcal G}}

\begin{document}

\title{Complex interpolation of smoothness Triebel-Lizorkin-Morrey spaces}

\author{D. I. Hakim$^{1}$, T. Nogayama$^{2}$ and Y. Sawano$^{3}$
}
\affil{
Department of Mathematics and Information Sciences, 
Tokyo Metropolitan University, 
1-1 Minami-Osawa,
	Hachioji-shi, Tokyo, 192-0397, Japan \\
Email: $^1$dennyivanalhakim@gmail.com, $^{2}$toru.nogayama@gmail.com, 
$^{3}$yoshihiro-sawano@celery.ocn.ne.jp	
}
\maketitle

\begin{abstract}
This paper extends the result in \cite{HNS15}
to Triebel-Lizorkin-Morrey spaces
which contains $4$ parameters $p,q,r,s$.
This paper reinforces our earlier paper \cite{HNS15}
by 
Nakamura, the first and the third authors
in two different directions.
First,
we include
the smoothness parameter $s$
and
the second smoothness parameter $r$.
In \cite{HNS15}
we assumed $s=0$ and $r=2$.
Here we relax the conditions on $s$ and $r$
to $s \in {\mathbb R}$ and $1 < r \le \infty$.
Second,
we apply a formula obtained by Bergh in 1978
to prove our main theorem without using the underlying sequence spaces.
\end{abstract}

{\bf Classification: 42B35, 41A17, 26B33}


\section{Introduction}

In \cite{YSY15},
Yuan, Sickel and Yang defined
the diamond subspace of the smoothness Morrey spaces.
We aim to decribe the complex interpolation
of a class of subspaces of smoothness Morrey spaces
defined in \cite{YSY15},
which extend the results in \cite{HNS15}.
Let $1< q \le p < \infty$.
For an $L^q_{\rm loc}$-function $f$,
its Morrey norm is defined by:
\begin{equation}\label{eq:130709-1}
\| f \|_{{\mathcal M}^p_q}
:= 
\sup_{x \in {\mathbb R}^n, \, R>0}
|B(x,R)|^{\frac{1}{p}-\frac{1}{q}}
\left(
\int_{B(x,R)}|f(y)|^q\,dy
\right)^\frac{1}{q},
\end{equation}
where $B(x,R)$ denotes the ball centered 
at $x \in {\mathbb R}^n$ of radius $R>0$.
The Morrey space ${\mathcal M}^p_q$
is the set of all $L^q$-locally integrable functions $f$
for which the norm $\| f \|_{{\mathcal M}^p_q}$ is finite.
We recall the definition of
Triebel-Lizorkin-Morrey spaces  as follows.
Let $1<q \le p<\infty$, $1 \le r \le \infty$ and $s \in {\mathbb R}$.
Choose
$\varphi \in {\mathcal S}$
so that
$
\chi_{B(2)} \le \varphi \le \chi_{B(3)}$
holds.
Set
\begin{equation}\label{eq:151113-22}
\varphi_0:=\varphi
\end{equation}
and 
\begin{equation}\label{eq:151113-21}
\varphi_j:=\varphi(2^{-j}\cdot)-\varphi(2^{-j+1}\cdot)
\end{equation}
for $j\in\N$. Note that $\varphi_j$ satisfies 
\begin{equation}
\sum_{j=0}^\infty\varphi_j=1.
\end{equation}
Next, we write 
\[
\varphi_j(D)g=\mathcal{F}^{-1}(\varphi_j \mathcal{F}g)
\]
where $\mathcal{F}$ and $\mathcal{F}^{-1}$ denote the Fourier transform and its inverse, defined by
\begin{equation*}
\begin{cases}
\displaystyle
\mathcal{F}g(\xi):=
(2\pi)^{-\frac{n}{2}}\int_{\R^n}g(x)e^{-ix\cdot\xi}dx\quad&(\xi\in\R^n)\\
\displaystyle
\mathcal{F}^{-1}g(x):=
(2\pi)^{-\frac{n}{2}}\int_{\R^n}g(\xi)e^{ix\cdot\xi}d\xi\quad&(x\in\R^n),
\end{cases}
\end{equation*} 
for $g\in L^1(\R^n)$.
Now, for $f \in {\mathcal S}'$,
we define 
\begin{align}
\label{eq:161114-104}
\| f \|_{{\mathcal E}^s_{pqr}}
&:=
\|\varphi_0(D)f\|_{{\mathcal M}^p_q}
+
\left\|
\left(
\sum_{j=1}^\infty
2^{jrs}| \varphi_j(D)f|^r
\right)^{\frac1r}
\right\|_{{\mathcal M}^p_q}.
\end{align}
The {\it Triebel-Lizorkin-Morrey space} 
${\mathcal E}^s_{pqr}$
is the set
of all
$f \in {\mathcal S}'$
for which
the norm
$\| f \|_{{\mathcal E}^s_{pqr}}$
is finite.
The parameters $r$ and $s$ are sometimes called 
the second smoothness parameter
and
the smoothness parameter,
respectively. 
Remark that the definition of ${\mathcal E}^s_{pqr}$ does not depend on the choice of the function $\varphi$ (see \cite[Theorem 1.4]{NNS15} or \cite{TaXu05}).

We are interested in the following closed subspace
of ${\mathcal E}^s_{p q r}$:
\begin{definition}{\rm \cite[Definition 2.23]{YSY15}}{(smoothness space)}
Let $1< q \le p < \infty$ and $1 \le r \le \infty$.
The space
$\overset{\diamond}{{\mathcal E}}{}^s_{p q r}$
denotes the closure with respect to
${\mathcal E}^s_{p q r}$
of the set of all smooth functions $f$
such that $\partial^\alpha f \in {\mathcal E}^s_{p q r}$
for all multi-indices $\alpha$.
\end{definition}

%

We characterize
$\overset{\diamond}{{\mathcal E}}{}^s_{p q r}$
in terms of the Littlewood-Paley decomposition,
which is a starting point of this paper.
\begin{theorem}\label{th-150928-2}
Let $1<q\leq p<\infty$, $1\le r \le \infty$, and 
$f \in {\mathcal E}^s_{p q r}$. 
Then $f$ is in
$\overset{\diamond}{{\mathcal E}}{}^s_{p q r}$,
if and only if 
$\sum_{j=0}^N\varphi_j(D)f$ 
converges to $f$ as 
$N\to \infty$ 
in 
${\mathcal E}^s_{p q r}$.
\end{theorem}

We seek to describe
the first and the second complex interpolation spaces
of
$\overset{\diamond}{{\mathcal E}}{}^{s_0}_{p_0 q_0 r_0}$
and
$\overset{\diamond}{{\mathcal E}}{}^{s_1}_{p_1 q_1 r_1}$,
where
the parameters $p_0, p_1, q_0, q_1 ,r_0, r_1$ satisfy
\begin{equation}\label{eq:151113-1}
p_0>p_1, \quad
1<q_0 \le p_0<\infty, \quad
1<q_1 \le p_1<\infty, \quad
1 <r_0 , r_1<\infty, \quad
\frac{p_0}{q_0}=\frac{p_1}{q_1}.
\end{equation}
Here, we may assume $p_0>p_1$ due to symmetry between $p_0$ and $p_1$.
To state our main result, we need the following notation.
Let $(X_0,X_1)$ be a compatible couple of Banach spaces.
Let $[X_0,X_1]_\theta$ and $[X_0,X_1]^\theta$ be 
the first
and second Calder\'{o}n complex  interpolation space
whose definition we recall in Section \ref{section:2}.
For $\theta \in (0,1)$,
define $p$, $q$, $r$ and $s$ by:
\begin{equation}\label{eq:151113-2}
\frac{1}{p}:=\frac{1-\theta}{p_0}+\frac{\theta}{p_1}, \quad
\frac{1}{q}:=\frac{1-\theta}{q_0}+\frac{\theta}{q_1}, \quad
\frac{1}{r}:=\frac{1-\theta}{r_0}+\frac{\theta}{r_1}, \quad
s:=(1-\theta)s_0+\theta s_1.
\end{equation}
A direct consequence of 
(\ref{eq:151113-1}) and (\ref{eq:151113-2})
is
\begin{equation}\label{eq:151113-3}
\frac{p}{q}=\frac{p_0}{q_0}=\frac{p_1}{q_1}.
\end{equation}
For $f \in {\mathcal S}'$ ,
we define
\[
S(f;r,s)
:=
\left(\sum_{j=0}^\infty|2^{j s}\varphi_j(D)f|^r\right)^{\frac1r}, \quad
\]
\[
S(f;a,J,r,s):=
\chi_{[a,a^{-1}]}(S(f;r,s))
\left(\sum_{j=J}^\infty|2^{j s}\varphi_j(D)f|^r\right)^{\frac1r}.
\]
Based on this notation, we state our main results as follows:
\begin{theorem}\label{thm:151113-2}
	Suppose that
	we have $13$ parameters
	$p_0,p_1,p,q_0,q_1,q,r,r_0,r_1,s,s_0,s_1$, and $\theta$
	satisfying 
	$(\ref{eq:151113-1})$ and $(\ref{eq:151113-2})$.
	\begin{enumerate}
		\item
		We have
		\begin{equation}\label{eq:151113-11}
		[
		\overset{\diamond}{{\mathcal E}}{}^{s_0}_{p_0q_0r_0},
		\overset{\diamond}{{\mathcal E}}{}^{s_1}_{p_1q_1r_1}]_\theta
		=
		\overset{\diamond}{{\mathcal E}}{}^{s}_{pqr}
		\cap
		[
		{\mathcal E}^{s_0}_{p_0q_0r_0},
		{\mathcal E}^{s_1}_{p_1q_1r_1}]_\theta.
		\end{equation}
		\item
		If $r_0=r_1$ and $s_0=s_1$, then 
		\begin{equation}\label{eq:151113-12}
		[\overset{\diamond}{{\mathcal E}}{}^{s_0}_{p_0q_0r_0},
		\overset{\diamond}{{\mathcal E}}{}^{s_1}_{p_1q_1r_1}]^\theta
		=
		\bigcap_{0<a<1}
		\left\{
		f \in {\mathcal E}^s_{p q r}
		\,:\,
		\lim_{J \to \infty}
		\|S(f;a,J,r,s)\|_{{\mathcal M}^p_q}
		=0
		\right\}.
		\end{equation}
	\end{enumerate}
\end{theorem}

\begin{theorem}\label{thm:170426s-2}
	Suppose that
	we have $13$ parameters
	$p_0,p_1,p,q_0,q_1,q,r,r_0,r_1,s,s_0,s_1$, and $\theta$
	satisfying \eqref{eq:151113-1} and \eqref{eq:151113-2}.
	Then
	we have
	\begin{equation}\label{eq:170426s-12}
	[{\mathcal E}^{s_0}_{p_0q_0r_0},
	{\mathcal E}^{s_1}_{p_1q_1r_1}]^\theta
	={\mathcal E}^{s}_{pqr}
	\end{equation}
	with equivalence of norms.
\end{theorem}

Having stated the main result in this paper,
let us investigate its relation with
the existing results.
The corresponding results
for the first complex interpolation of Triebel-Lizorkin-Morrey spaces 
was obtained by Yang, Yuan and Zhuo 
(see \cite[Corollary 1.11]{YYZ13}). They proved the following theorem.
\begin{theorem}{\rm \cite[Corollary 1.11]{YYZ13}}
	Suppose that
	we have $13$ parameters
	$p_0$, $p_1$, $p$, $q_0$, 
	$q_1$, $q$, $r$, $r_0$, $r_1$, $s$, $s_0$, $s_1$, 
	and $\theta$
	satisfying 
	$(\ref{eq:151113-1})$ and $(\ref{eq:151113-2})$. Then
	\begin{equation}\label{eq:151113-11y-2}
	[
	{\mathcal E}^{s_0}_{p_0q_0r_0},
	{\mathcal E}^{s_1}_{p_1q_1r_1}]_\theta
	\subseteq
	{\mathcal E}^s_{p q r}.
	\end{equation}
\end{theorem}
Remark that \eqref{eq:151113-11y-2} will be used 
in the proof of Theorem \ref{thm:170426s-2}.
As a corollary of 
(\ref{eq:Bergh}) to follow
and Theorem \ref{thm:170426s-2},
we have 
the corresponding result
for the first complex interpolation of Tribel-Lizorkin-Morrey spaces 
which refines \eqref{eq:151113-11y-2}.
\begin{theorem}\label{thm:151113-2a}
	Suppose that
	we have $13$ parameters
	$p_0$, $p_1$, $p$, $q_0$, 
$q_1$, $q$, $r$, $r_0$, $r_1$, $s$, $s_0$, $s_1$, 
and $\theta$
	satisfying 
	$(\ref{eq:151113-1})$ and $(\ref{eq:151113-2})$. Then
		\begin{equation}\label{eq:151113-11y}
		[
		{\mathcal E}^{s_0}_{p_0q_0r_0},
		{\mathcal E}^{s_1}_{p_1q_1r_1}]_\theta
		=
		\overline
		{
		{\mathcal E}^{s_0}_{p_0q_0r_0}
                \cap
		{\mathcal E}^{s_1}_{p_1q_1r_1}}^{{\mathcal E}^s_{p q r}}.
		\end{equation}
\end{theorem}
Meanwhile,
Nakamura, the first and the third authors
obtained the description of the interpolation of
diamond Morrey spaces in \cite{HNS15},
which we describe below.
Let $1< q \le p < \infty$.
The space
$\overset{\diamond}{{\mathcal M}}{}^p_{q}$
denotes the closure with respect to
${\mathcal M}^p_{q}$
of the set of all smooth functions $f$
such that $\partial^\alpha f \in {\mathcal M}^p_{q}$
for all multi-indices $\alpha$
\cite{YSY15}.

Due to the result by Mazzucato \cite[Proposition 4.1]{Mazzucato01},
we see that
\[
{\mathcal M}^p_q={\mathcal E}^0_{pq2}.
\]
Thus, $\overset{\diamond}{\mathcal E}{}^0_{pq2}=
\overset{\diamond}{\mathcal M}{}^p_q$
with norm equivalence
and 
Theorem \ref{thm:151113-2} recaptures
the interpolation of
$\overset{\diamond}{\mathcal M}{}^{p_0}_{q_0}$
and
$\overset{\diamond}{\mathcal M}{}^{p_1}_{q_1}$
as the special case of
$r_0=r_1=r=2$
and
$s_0=s_1=s=0$. Thus, 
we see that
Theorem \ref{thm:151113-2} extends
\cite[Theorem 1.9]{HNS15}

One of the difficulties in dealing with the space
$\overset{\diamond}{{\mathcal M}}{}^p_q$
with $1<q<p<\infty$
is that this closed subspace does not enjoy
the lattice property unlike many other important subspaces
defined in
\cite{CDM10,ST-Tokyo,YSY15}.

Let us now recall some progress in interpolation theory
of Morrey spaces. 
The earlier result about the interpolation of Morrey spaces 
can be traced back in \cite{Stampacchia65}.
In \cite[p. 35]{CPP98}
Cobos, Peetre, and Persson pointed out that
\[
[{\mathcal M}^{p_0}_{q_0},{\mathcal M}^{p_1}_{q_1}]_\theta
\subset
{\mathcal M}^p_q
\]
whenever
$1 \le q_0 \le p_0<\infty$,
$1 \le q_1 \le p_1<\infty$,
and
$1 \le q \le p<\infty$
satisfy
\begin{equation}\label{eq:150531-1}
\frac{1}{p}=\frac{1-\theta}{p_0}+\frac{\theta}{p_1}, \quad
\frac{1}{q}=\frac{1-\theta}{q_0}+\frac{\theta}{q_1}.
\end{equation}
A counterexample by Blasco, Ruiz, and Vega \cite{BRV99,RuVe95}, shows that if 
we assume 
(\ref{eq:150531-1}) only, then there exists a bounded linear operator $T$
from $\cM^{p_k}_{q_k}(\R^n)$ ($k=0,1$) to $L^1(\R^n)$, but $T$ is unbounded from $\cMpq(\R^n)$ to $L^1(\R^n)$.
By using the counterexample by Ruiz and Vega in  \cite{RuVe95}, 
Lemari\'e-Rieusset \cite[Theorem 3(ii)]{Lemarie13} 
showed that
if an interpolation functor $F$ satisfies
$
F[{\mathcal M}^{p_0}_{q_0},{\mathcal M}^{p_1}_{q_1}]
=
{\mathcal M}^p_q
$
under the condition (\ref{eq:150531-1}),
then
\begin{equation}\label{eq:150531-2}
\frac{q_0}{p_0}=\frac{q_1}{p_1}
\end{equation}
holds.
Lemari\'e-Rieusset \cite{Lemarie13,Lemarie14} also showed that 
Morrey space is closed under
the second complex interpolation method, namely, 
\begin{equation}\label{eq:170427-1}
[{\mathcal M}^{p_0}_{q_0},{\mathcal M}^{p_1}_{q_1}]^\theta
		={\mathcal M}^p_q.
\end{equation}
Meanwhile, as for the interpolation result
under (\ref{eq:150531-1}) and (\ref{eq:150531-2})
by using the first Calder\'{o}n's complex interpolation functor,
Lu, Yang, and Yuan obtained the following description:
\begin{align}\label{eq:170402}
[{\mathcal M}^{p_0}_{q_0},{\mathcal M}^{p_1}_{q_1}]_\theta
=
\overline{{\mathcal M}^{p_0}_{q_0} \cap {\mathcal M}^{p_1}_{q_1}}^{{\mathcal M}^p_q}
\end{align}
in \cite[Theorem 1.2]{LuYangYuan13}.
Their result is 
in the setting of a metric measure space.
The generalization of 
the result of Lu et. al and Lemari\'e-Rieusset
in the setting of generalized Morrey spaces and 
generalized Orlicz-Morrey spaces can be seen in 
\cite{HS}. 
 The first and third authors \cite{HS2} also obtain 
a refinement of  \eqref{eq:170402}  as follows: 
\begin{equation}\label{eq:170427-2}
[{\mathcal M}^{p_0}_{q_0},{\mathcal M}^{p_1}_{q_1}]_\theta
=
\left\{f \in {\mathcal M}^p_q\,:\,
\lim\limits_{a \to 0^+}\|(1-\chi_{[a,a^{-1}]}(|f|))f\|_{{\mathcal M}^p_q}=0
\right\}.
\end{equation}

The complex interpolation of variable exponent Morrey spaces can be seen \cite{MRZ16}. 
As for the real interpolation results,
Burenkov and Nursultanov obtained
an interpolation result in local Morrey spaces \cite{BuNu10}
and  
their results are generalized
by Nakai and Sobukawa 
to $B^u_w$ setting \cite{NS}.	
In \cite{YYZ13},
Yang, Yuan, and Zhuo considered
the interpolation of smoothness Morrey spaces
considered in
\cite{HaSk14,HMS16,Ho11-1,Mazzucato01,NNS15,Sawano08-2,SaTa07-1,TaXu05,YaYu08-2,YaYu10-1,YSY10,YWY10,YSY15}.

Compared to the work \cite{YYZ13},
we believe that the main tool
is Lemma \ref{lem:151112-1},
where the function $``\log"$ plays the key role.
An experience obtained in \cite{HS}
shows that the function $\lq\lq\log"$ is essential
when we consider the complex interpolation functor.

Let us explain why the interpolation of Morrey spaces
are complicated unlike Lebesgue spaces.
From
(\ref{eq:170427-1})
and
(\ref{eq:170427-2})
we learn that 
the first complex interpolation functor behaves differently from Lebesgue spaces.
This problem comes basically from the fact that
the Morrey norm ${\mathcal M}^p_q$
involves the supremum over all balls $B(x,R)$.
Due to this fact, we have many difficulties when $1<q<p<\infty$, namely:
\begin{enumerate}
\item

The Morrey space ${\mathcal M}^p_q$ 
is not included in $L^1+L^\infty$;
see \cite[Section 6]{HS2}.
\item
The Morrey space ${\mathcal M}^p_q$ is not reflexive; see \cite[Example 5.2]{ST-Tokyo}
and \cite[Theorem 1.3]{YaYu11}.
\item
Let
$p_0,p_1,p,q_0,q_1,q$
satisfy
$(\ref{eq:151113-1})$.
Let $q<\tilde{q}<p$.
The spaces
	$C^\infty_{\rm c}$,
${\mathcal M}^p_{\tilde{q}}$,
${\mathcal M}^{p_0}_{q_0} \cap {\mathcal M}^{p_1}_{q_1}$
are
not dense in 
 ${\mathcal M}^p_q$;
see \cite[Proposition 2.16]{Triebel14},
\cite{Sawano17}
and
\cite{HS,YSY15},
respectively.
\item
The Morrey space ${\mathcal M}^p_q$ is not separable; 
see \cite[Proposition 2.16]{Triebel14}.
\end{enumerate}
These facts prevent us from
using many theorems
in the textbook in \cite{BeLo76}.

We organize the remaining part of this paper as follows:
Section \ref{section:2} collects some preliminary facts
such as the property of the complex interpolation
and the maximal inequalities for Morrey spaces.
We prove Theorems \ref{thm:151113-2} and \ref{thm:170426s-2}
in Section \ref{section:3}
except a key fact on $G$
defined
in Section \ref{section:3}.
This fact will be proved in Section \ref{ss331}.

\section{Preliminaries}
\label{section:2}

\subsection{Complex interpolation functors}

Let $E$ be a subset of ${\mathbb C}$ and $X$ be a Banach space,
and define
\begin{equation}
{\rm BC}(E,X)
:= 
\left\{ f : E \to X \, : \,
f \mbox{ is continous and satisfies }
\sup_{z \in E}\| f(z) \|_{X}<\infty
\right\}.
\end{equation}
If $E$ is an open set in $\C$,
then ${\mathcal O}(E,X)$ denotes the set of all holomorphic functions
on $E$ whose value assumes $X$.
\begin{definition}\label{defi:150824-20}
Let
$U:=  \{ z \in {\mathbb C} \, : \, 0 < {\rm Re}\,(z) < 1 \}$
and
$\overline{U}$ be its closure.
\end{definition}

We recall the definition of the complex interpolation functors as follows:
\begin{definition}[Calder\'{o}n's first complex interpolation space, \cite{BeLo76, Calderon3}]\label{defi:150824-19}
Let $(X_0,X_1)$ be a compatible couple 
of Banach spaces.
\begin{enumerate}
\item
The space
${\mathcal F}(X_0,X_1)$ is defined as the set 
of all functions $F :\bar{U} \to X_0+X_1$
such that 
\begin{enumerate}
\item $F\in {\rm BC}(\overline{U},X_0+X_1)$,
\item $F \in {\mathcal O}(U,X_0+X_1)$,
\item the functions $t \in {\mathbb R} \mapsto F(j+it) \in X_j$ 
are bounded and continuous on ${\mathbb R}$ for $j=0,1$.
\end{enumerate}
The space ${\mathcal F}(X_0,X_1)$ is equipped with the norm
\begin{align*}
\|F\|_{{\mathcal F}(X_0,X_1)} 
:=  
\max \left\{ 
\sup\limits_{t\in {\mathbb R}} \|F(it)\|_{X_0}
, \
\sup\limits_{t\in {\mathbb R}} \|F(1+it)\|_{X_1}
\right\}.
\end{align*}
\item
Let $\theta \in (0,1)$.
Define the complex interpolation space $[X_0,X_1]_{\theta}$ with respect to $(X_0,X_1)$ to be the set 
of all functions $f\in X_0+X_1$ such
that $f=F(\theta)$ for some $F\in {\mathcal F}(X_0,X_1)$.
The norm on $[X_0,X_1]_{\theta}$ is defined by 
\begin{align*}
\|f\|_{[X_0,X_1]_{\theta}} :=   
\inf\{ \|F\|_{{\mathcal F}(X_0,X_1)} : f=F(\theta) \mathrm{\ for \ some \ } F \in {\mathcal F}(X_0,X_1) \}.
\end{align*}
\end{enumerate}
\end{definition}
According to \cite{Calderon3},
$[X_0,X_1]_{\theta}$ is a Banach space. See also \cite[Theorem 4.1.2]{BeLo76}.


Now, we recall the definition of
Calder\'{o}n's second complex interpolation space.
Let $X$ be a Banach space.
The space ${\rm Lip}({\mathbb R},X)$ is defined 
to be the set of all functions $F:{\mathbb R} \to X$
for which 
\[
\|F\|_{{\rm Lip}({\mathbb R},X)}
:=  
\sup_{-\infty<s<t<\infty}
\frac{\|F(t)-F(s)\|_X}{|t-s|}<\infty.
\]
\begin{definition}[Calder\'{o}n's second complex interpolation space, \cite{BeLo76, Calderon3}]
Suppose that we have a pair $(X_0,X_1)$ is a compatible couple of 
Banach spaces.
\begin{enumerate}
\item
Define 
${\mathcal G}(X_0,X_1)$ as the set 
of all functions $G:\overline{U} \to X_0+X_1$
such that 
\begin{enumerate}
\item 
$G$ is continuous on $\overline{U}$ and 
$\sup\limits_{z\in \overline{U}} 
\left\|\frac{G(z)}{1+|z|}\right\|_{X_0+X_1}<\infty$,
\item $G$ is holomorphic in $U$,
\item the functions 
\[
t \in {\mathbb R} \mapsto G(j+it)-G(j) \in X_j
\]
are Lipschitz continuous on ${\mathbb R}$ for $j=0,1$.
\end{enumerate}
The space ${\mathcal G}(X_0,X_1)$ is equipped with the norm
\begin{align}\label{eq:norm-G}
\|G\|_{{\mathcal G}(X_0,X_1)} 
:=  
\max \left\{ \|G(i\cdot)\|_{{\rm Lip}({\mathbb R}, X_0)}, \
\|G(1+i\cdot)\|_{{\rm Lip}({\mathbb R}, X_1)}
\right\}.
\end{align}
\item 
Let $\theta \in (0,1)$.
Define the complex interpolation space $[X_0,X_1]^{\theta}$ with respect to $(X_0,X_1)$ to be the set 
of all functions $f\in X_0+X_1$ such
that 
\begin{equation}\label{eq:170401-1}
f=G'(\theta)=\lim_{h \to 0}\frac{G(\theta+h)-G(\theta)}{h}
\end{equation} 
for some $G\in {\mathcal G}(X_0,X_1)$.
The norm on $[X_0,X_1]^{\theta}$ is defined by 
\begin{align*}
\|f\|_{[X_0,X_1]^{\theta}} :=   
\inf\{ \|G\|_{{\mathcal G}(X_0,X_1)} : f=G'(\theta) \mathrm{\ for \ some \ } G \in {\mathcal G}(X_0,X_1) \}.
\end{align*}
The space $[X_0,X_1]^{\theta}$
is called Calder\'{o}n's second complex interpolation space,
or the upper complex interpolation space
of $(X_0,X_1)$.
\end{enumerate}\label{defi:150824-18}
\end{definition}

One of the fundamental relations
between the first and the second complex interpolation is as follows:
\begin{equation}\label{eq:Bergh}
[X_0,X_1]_\theta=\overline{X_0 \cap X_1}{}^{[X_0,X_1]^\theta}
\end{equation}
according to the result by Bergh \cite{Be79}.
This relation explains
why we start by calculating the second interpolation
in the proof of Theorems \ref{thm:170426s-2}
and \ref{thm:151113-2a}.

If we combine Lemmas \ref{lem:dense0} and \ref{lem:dense1}
below,
we see that (\ref{eq:Bergh}) follows.
\begin{lemma}{\rm \cite{Be79}}
\label{lem:dense0}
Let $x \in X_0 \cap X_1$.
Then
$\|x\|_{[X_0,X_1]^\theta}=\|x\|_{[X_0,X_1]_\theta}$.
\end{lemma}

\begin{lemma}
{\rm \cite[Theorem 4.22 (a)]{BeLo76}}
\label{lem:dense1}
The space $X_0\cap X_1$ is dense in $[X_0,X_1]_\theta$.
\end{lemma}

A direct consequence of Lemma \ref{lem:dense1} is:
\begin{lemma}\label{lem:dense2}
$[X_0,X_1]^\theta \subseteq 
\overline{X_0\cap X_1}^{X_0+X_1}$.
\end{lemma}

\begin{proof}
We observe that
$[X_0,X_1]^\theta \subset \overline{[X_0,X_1]_\theta}^{X_0+X_1}$
from the definition of $[X_0,X_1]^\theta$;
see (\ref{eq:170401-1}).
In fact, for $f\in [X_0, X_1]^\theta$, 
there exists $G\in \cG(X_0,X_1)$ such that $f=G'(\theta)$. We define 
\[
F_j(z):=\frac{G(z+ij^{-1})-G(z)}{ij^{-1}}
\]
for $j\in \N$ and $z\in \overline{S}$. 
Then, $F_j(\theta) \in [X_0,X_1]_\theta$ and according to (\ref{eq:170401-1}), we have $f\in \overline{[X_0,X_1]_\theta}^{X_0+X_1}$.

Since
$[X_0,X_1]_\theta=\overline{X_0 \cap X_1}^{[X_0,X_1]_\theta}
\subset \overline{X_0 \cap X_1}^{X_0+X_1}$
from Lemma \ref{lem:dense1},
it follows that
$\overline{[X_0,X_1]_\theta}^{X_0+X_1}
\subset \overline{X_0 \cap X_1}^{X_0+X_1}$.
Putting together these observations,
we obtain the desired result.
\end{proof}

\subsection{Operators on Morrey spaces}

Let ${\mathcal B}$ denote the set of all balls in ${\mathbb R}^n$.
We recall the definition and the fundamental property of the Hardy-Littlewood maximal operator
$M$.
 \begin{definition}[Hardy-Littlewood maximal operator]
 For a measurable function $f$,
 define a function $Mf$ by:
 \begin{equation}\label{eq:maximal operator}
 M f(x):=
 \sup_{B \in {\mathcal B}}\frac{\chi_B(x)}{|B|}\int_B |f(y)|\,d y.
 \end{equation}
 The mapping $M:f \mapsto Mf$
 is called the Hardy-Littlewood maximal operator.
 \index{Hardy-Littlewood maximal operator@Hardy-Littlewood maximal operator}
 \label{defi:maximal}
 \end{definition}

 \begin{theorem}\label{thm:vec-maxi}
 {\rm \cite[Theorem 2.4]{SaTa05},
 {\rm \cite[Lemma 2.5]{TaXu05}}}
 Suppose that the parameters $p,q,r$ satisfy
 \[
 1<q \le p< \infty  \ {\rm and } \  1<r \le \infty.
 \]
 Then
 \begin{equation}\label{eq:150928-71}
 \left\|\left(\sum_{j=1}^\infty (M f_j){}^r\right)^{\frac1u}
 \right\|_{{\mathcal M}^p_q}\lesssim
 \left\|\left(\sum_{j=1}^\infty|f_j|^r\right)^{\frac1u}
 \right\|_{{\mathcal M}^p_q}
 \end{equation}
 for every sequence of measurable functions $\{f_j\}_{j=0}^\infty$.
 When $r=\infty$, then $(\ref{eq:150928-71})$ reads;
 \begin{equation}
 \left\|\sup_{j \in {\mathbb Z}} Mf_j
 \right\|_{{\mathcal M}^p_q}\lesssim
 \left\|\sup_{j \in {\mathbb Z}} |f_j|
 \right\|_{{\mathcal M}^p_q}.
 \end{equation}
 \end{theorem}

	As a direct consequence of Theorem \ref{thm:vec-maxi}, we have the following lemma.
	\begin{lemma}\label{lem:161017-1}
		Let $1<q\le p<\infty$, $1< r \le \infty$, and $J\in \N $.
		Let $\{g_j\}_{j=J}^\infty$ be a sequence of measurable functions such that 
		\[
		\left\| 
		\left(
		\sum_{j=J}^{\infty}
		|g_j|^r
		\right)^{\frac{1}{r}}
		\right\|_{\cM^p_q}<\infty.
		\]  
		Then
		\begin{align}\label{eq:161017-1}
		\left\|
		\left(
		\sum_{l=1}^{\infty}
		\left|
		\varphi_l(D)
		\left(
		\sum_{j=J}^{\infty}
		\varphi_j(D)g_j
		\right)
		\right|^r
		\right)^{\frac{1}{r}}
		\right\|_{\cM^p_q}
		\lesssim
		\left\| 
		\left(
		\sum_{j=J}^{\infty}
		|g_j|^r
		\right)^{\frac{1}{r}}
		\right\|_{\cM^p_q}.
		\end{align}
	\end{lemma} 
	\begin{proof}
		Note that, for $f\in L^1_{\rm loc}(\R^n)$, we have
		\begin{align}\label{eq:161018-1}
		|\varphi_l(D)f|\lesssim Mf.
		\end{align}
		We use \eqref{eq:161018-1} and the fact that $\varphi_l \varphi_j=0$ whenever $|l-j|\ge 2$
		to obtain
		\begin{align}\label{eq:170323-1}
		\sum_{l=1}^{\infty}
		\left|
		\varphi_l(D)
		\left(
		\sum_{j=J}^{\infty}
		\varphi_j(D)g_j
		\right)
		\right|^r
		&\le 
		\sum_{l=J-1}^{\infty}
		\left| 
		\sum_{j=\max(l-1,J)}^{l+1}
		\varphi_l(D) \varphi_j(D) g_j
		\right|^r
		\nonumber
		\\
		&\lesssim
		\sum_{l=J-1}^{\infty}
		\sum_{j=\max(l-1,J)}^{l+1}
		|\varphi_l(D)[\varphi_j(D)g_j]|^r
		\nonumber
		\\
		&\lesssim
		\sum_{j=J}^{\infty}
		M(\varphi_j(D)g_j)^r.
		\end{align}
		By combining \eqref{eq:161018-1}, \eqref{eq:170323-1}, and Theorem \ref{thm:vec-maxi} ,  we have
		\begin{align*}
		&\left\|
		\left(
		\sum_{l=1}^{\infty}
		\left|
		\varphi_l(D)
		\left(
		\sum_{j=J}^{\infty}
		\varphi_j(D)g_j
		\right)
		\right|^r
		\right)^{\frac{1}{r}}
		\right\|_{\cM^p_q}
		\lesssim
		\left\|
		\left(
		\sum_{j=J}^{\infty}
		M(\varphi_j(D)g_j)^r
		\right)^{\frac{1}{r}}
		\right\|_{\cM^p_q}
                \\
		&\lesssim
		\left\|
		\left(
		\sum_{j=J}^{\infty}
		|\varphi_j(D)g_j|^r
		\right)^{\frac{1}{r}}
		\right\|_{\cM^p_q}
		\lesssim
		\left\|
		\left(
		\sum_{j=J}^{\infty}
		|Mg_j|^r
		\right)^{\frac{1}{r}}
		\right\|_{\cM^p_q}
		\lesssim
		\left\|
		\left(
		\sum_{j=J}^{\infty}
		|g_j|^r
		\right)^{\frac{1}{r}}
		\right\|_{\cM^p_q}.
		\end{align*}
	\end{proof}

\subsection{Some inequalities}

We use the following inequality
which improves slightly the one in \cite{Triebel83}.
\begin{lemma}\label{lem:151109-1}{\rm \cite[Lemma 2.17]{NoiSawano2012}}
Fix $J\in \Z \cup\{-\infty\}$.
Let $\{a_j\}_{j=J}^\infty$ be a non-negative sequence and 
$\kappa>0$.
Then
\[
\sum_{j=J}^\infty a_j\left(\sum_{k=J}^j a_k\right)^{\kappa-1}
\le
\frac{1}{\min(\kappa,1)}
\left(\sum_{j=J}^\infty a_j\right)^\kappa.
\]
Here, we assume there is a non-zero $a_j$.
\end{lemma}

When we consider the complex interpolation
of the second kind of classical Morrey spaces, 
we are faced with the function
$|\log|f||^{-1}$ in the proof;
see \cite{HS}.
To take an advantage of this $\lq\lq\log"$ factor fully,
we will use the following series of lemmas:

\begin{lemma}\label{lem:161109-1}
Let $1\le r<\infty$ and $z\in \C$ be such that ${\rm Re}(z)\ge 0$.
Then there exists a constant $C=C_z>0$ such that 
\begin{align}\label{eq:160111-1}
\left|
\frac{s^{z}-1}{\log(s^r)}
\right|
\le
C\left(\log\left(s+\frac{1}{s}\right)\right)^{-1}
\end{align}
for every $s\in (0,1)$ and 
\begin{align}\label{eq:160111-2}
\left|
\frac{s^{-z}-1}{\log(s^r)}
\right|
\le
C\left(\log\left(s+\frac{1}{s}\right)\right)^{-1}
\end{align}
for every $s>1$.
\end{lemma}

\begin{lemma}\label{lem:161101-1}
Let $1\le r<\infty$ and fix $t\in \R$. Then there exists a constant $C_t>0$ such that
\begin{align}\label{eq:160108-1}
\left|
\frac{s^{it}-1}{\log(s^r)}
\right|
\le 
C_t\left(\log\left(s+\frac1s\right)\right)^{-1},
\end{align}
for all $s>0$ with $s\neq 1$.
\end{lemma}

As we have mentioned,
the function of the form
$|\log|f||^{-1}$ plays a crucial role
for later considerations.
We will need some variant including
the logarithm.
We use the functions defined by
\begin{equation}\label{eq:160522-1}
\Phi_{\kappa}(t):= 
t^{\kappa-1}\left(\log\left(t+\frac1t \right)\right)^{-1}, \quad
\Psi_{\kappa}(t):= 
\int_0^t \Phi_{\kappa}(\sqrt[r]{s})^r\,ds \quad (t,\kappa>0, \  1\le r<\infty ).
\end{equation}

\begin{lemma}\label{lem:151112-1}
Let $1\le r<\infty$ and $\kappa>0$.
Then we have
\begin{equation}\label{eq:151111-3}
\sum_{j=0}^\infty \left[ a_j
\Phi_{\kappa}\left(\sqrt[r]{\sum_{k=0}^{j}a_k{}^r}
\right)\right]^r
\lesssim
\Psi_{\kappa}\left(\sum_{j=0}^{\infty}a_j{}^r
\right)
\end{equation}
for all nonnegative square summable sequences
$\{a_j\}_{j=0}^\infty$.
\end{lemma}

\begin{proof}
Assume first that $\kappa\in (0,1)$.
In this case,
\begin{equation}\label{eq:160523-51}
\Phi_\kappa(t_1)\gtrsim \Phi_\kappa(t_2)
\end{equation} for every $t_1\le t_2$.
We observe
\begin{align*}
\Psi_{\kappa}
\left( \sum_{j=0}^\infty a_j{}^r \right)
&=
\int_0^{\sum_{j=0}^\infty a_j{}^r }
\Phi_{\kappa}(\sqrt[r]{s})^r \ ds\\
&=
\int_0^{a_0{}^r}
\Phi_{\kappa}(\sqrt[r]{s})^r \ ds+
\sum_{j=0}^\infty
\int_{\sum_{k=0}^j a_k{}^r }^{\sum_{k=0}^{j+1} a_k{}^r }
\Phi_{\kappa}(\sqrt[r]{s})^r \ ds.
\end{align*}
Using (\ref{eq:160523-51}), we have
\begin{align*}
\Psi_{\kappa}
\left( \sum_{j=0}^\infty a_j{}^r \right)
&\gtrsim
a_0{}^r \Phi_{\kappa}\left(\sqrt[r]{a_0{}^r}\right)^r
+
\sum_{j=0}^\infty
a_{j+1}{}^r \Phi_{\kappa}
\left(\sqrt[r]{\sum_{k=0}^{j+1} a_k{}^r}\right)^r\\
&=
\sum_{j=0}^\infty
a_{j}{}^r \Phi_{\kappa}
\left(\sqrt[r]{\sum_{k=0}^{j} a_k{}^r}\right)^r.
\end{align*}
For the case $\kappa>1$, 
observe that $\Phi_\kappa(t)$  satisfies 
\begin{align}\label{eq:160107-1}
\Phi_{\kappa}(2t) \le 2^{\kappa} \Phi_{\kappa}(t)
\end{align}
for all $t>0$.
In addition, we also can choose 
$C_2>0$ such that 
\begin{align}\label{eq:160107-2}
\Phi_{\kappa}(t_1) \le C_2 \Phi_{\kappa}(t_2)
\end{align}
for every $t_1\le t_2$.
Write $R:= \sum_{j=0}^\infty a_j{}^r$.
By combining \eqref{eq:160107-1}
and \eqref{eq:160107-2}, we get
\begin{align*}
\sum_{j=0}^\infty
\left[a_j 
\Phi_{\kappa}
\left(\sqrt[r]{\sum_{k=0}^j
a_k{}^r}\right)\right]^r
&\lesssim R\Phi_{\kappa}(\sqrt[r]{R})^r\\
&\lesssim
2^{2\kappa}
R\Phi_{\kappa}\left(\frac12
\sqrt[r]{R}\right)^r\\
&\lesssim
\int_{R/4}^R
\Phi_{\kappa}(\sqrt[r]{s})^r \ ds\\
&\le \Psi_\kappa\left(\sum_{j=0}^{\infty}a_j{}^r\right)
\end{align*}
as desired.
\end{proof}

\begin{lemma}\label{lem:161109-2}
Let $1\le r<\infty$,  $\kappa>0$ and $a\in (0,1)$.
Then, we have 
\begin{align}\label{eq:160111-3}
\Psi_{\kappa}(t^r)\le \frac{1}{\kappa (\log 2)^{r}}
\left(a^{(r-1)\kappa}+
\left(\log\left(\sqrt[r]{a}+
\frac{1}{\sqrt[r]{a}}\right)\right)^{-r}\right)t^{r\kappa},
\end{align}
for every $t\in (0,a)\cup (a^{-1},\infty)$.
\end{lemma}

\begin{proof}
By the fundamental theorem on calculus,
we have
\[
\Psi_{\kappa}(t^r)
=
\int_{0}^{t} 
s^{\kappa-1}\left(\log\left(\sqrt[r]{s}+\frac{1}{\sqrt[r]{s}}\right)\right)^{-r} \ ds
+
\int_{t}^{t^r} 
s^{\kappa-1}\left(\log\left(\sqrt[r]{s}+\frac{1}{\sqrt[r]{s}}\right)\right)^{-r} \ ds.
\]
For $t>a^{-1}$, we have
\begin{align*}
\Psi_{\kappa}(t^r)
&\le
\frac{1}{\left(\log 2 \right)^{r}} 
\int_{0}^{t} 
s^{\kappa-1} \ ds
+
\left(\log\left(\sqrt[r]{t}+\frac{1}{\sqrt[r]{t}}\right)\right)^{-r} 
\int_{t}^{t^r} 
s^{\kappa-1} \ ds\\
&\le
\frac{1}{\kappa (\log 2)^{r}} t^{\kappa}
+
\frac{1}{\kappa} 
\left(\log\left(\sqrt[r]{a}+\frac{1}{\sqrt[r]{a}}\right)\right)^{-r}t^{r\kappa}\\
&= \frac{1}{\kappa (\log 2)^{r}}
\left( \frac{1}{t^{(r-1)\kappa}}+
\left(\log 2 \cdot
\left(\log\left(\sqrt[r]{a}+\frac{1}{\sqrt[r]{a}}\right)\right)^{-1}
\right)^{r}\right)t^{r\kappa}\\
&\le \frac{1}{\kappa (\log 2)^{r}}
\left( a^{(r-1)\kappa}+
\left(\log\left(\sqrt[r]{a}+\frac{1}{\sqrt[r]{a}}\right)\right)^{-r}\right)t^{r\kappa}.
\end{align*}
Meanwhile, using
\[
\Psi_{\kappa}(t^r)
=\int_{0}^{t^r}
\Phi_{\kappa}(\sqrt[r]{s})^{r} \ ds
=\int_{0}^{t^r}
s^{\kappa-1}\left(\log\left(\sqrt[r]{s}+\frac{1}{\sqrt[r]{s}}\right)\right)^{-r} \ ds,
\]
we have
for $0<t<a$, we have
\begin{align*}
\Psi_{\kappa}(t^r)
&\le 
\left(\log\left(\sqrt[r]{t}+\frac{1}{\sqrt[r]{t}}\right)\right)^{-r}
\int_{0}^{t^r}
s^{\kappa-1} \ ds\\
&\le \frac{1}{\kappa (\log 2)^{r}}
t^{r\kappa}\left(\log\left(\sqrt[r]{a}+\frac{1}{\sqrt[r]{a}}\right)\right)^{-r}\\
&\le \frac{1}{\kappa (\log 2)^{r}}
\left( a^{(r-1)\kappa}+
\left(\log\left(\sqrt[r]{a}+\frac{1}{\sqrt[r]{a}}\right)\right)^{-r}\right)t^{r\kappa},
\end{align*}
as desired.
\end{proof}
For checking the holomorphicity of the second complex interpolation functor, 
we invoke the following lemma:
\begin{lemma}{\rm \cite[Lemma 3]{HS}}
Let $h \in {\mathbb C}$ and
$\varepsilon>0$.
Assume that $\varepsilon>2|h|$.
Then, there exists $C_\varepsilon>0$ such that
\begin{align}\label{eq:150514-3}
\sup_{0<t\le 1}
t^\varepsilon
\left|\frac{\exp(h\log t)-1}{h\log t}-1\right|
\le
C_\varepsilon |h|
\end{align}
and
\begin{align}\label{eq:150514-4}
\sup_{t>1}
t^{-\varepsilon}
\left|\frac{\exp(h \log t)-1}{h \log t}-1\right|
\le
C_\varepsilon |h|.
\end{align}
\end{lemma}

\begin{proposition}\label{cr-151002-1}
Suppose that
we have $13$ parameters
$$p_0,p_1,p,q_0,q_1,q,r,r_0,r_1,s,s_0,s_1,\theta$$
satisfying 
$(\ref{eq:151113-1})$ and $(\ref{eq:151113-2})$.
Then
$
\overset{\diamond}{{\mathcal E}}{}^{s_0}_{p_0q_0r_0}
\cap
\overset{\diamond}{{\mathcal E}}{}^{s_1}_{p_1q_1r_1}
\subset
\overset{\diamond}{{\mathcal E}}{}^{s}_{pqr}
$.
\end{proposition}

\begin{proof}
We take 
$f\in 
\overset{\diamond}{{\mathcal E}}{}^{s_0}_{p_0q_0r_0}
\cap
\overset{\diamond}{{\mathcal E}}{}^{s_1}_{p_1q_1r_1}
$.
Theorem \ref{th-150928-2}
implies that 
\begin{equation}\label{151002-1}
\left\|
f
-
\sum_{j=0}^N\varphi_j(D)f
\right\|_{{\mathcal E}^{s_k}_{p_kq_kr_k}}
\to
0
\end{equation}
as $N\to\infty$ for $k=0,1$.
By the H\"{o}lder inequality,
we have
\begin{align}\label{eq:170419-1}
\left\|
f
-
\sum_{j=0}^N\varphi_j(D)f
\right\|_{{\mathcal E}^{s}_{pqr}}
\le \left\|
f
-
\sum_{j=0}^N\varphi_j(D)f
\right\|_{{\mathcal E}^{s_0}_{p_0q_0r_0}}^{1-\theta}
\left\|
f
-
\sum_{j=0}^N\varphi_j(D)f
\right\|_{{\mathcal E}^{s_1}_{p_1q_1r_1}}^\theta.
\end{align}
Combining \eqref{151002-1} and \eqref{eq:170419-1},
we obtain the desired result.
\end{proof}

\section{Proofs}
\label{section:3}

\subsection{Proof of Theorem \ref{thm:170426s-2}}
 According to \cite[Corollary 1.11]{YYZ13},
 we have
 \begin{equation}\label{eq:170426s-11}
 [
 {\mathcal E}^{s_0}_{p_0q_0r_0},
 {\mathcal E}^{s_1}_{p_1q_1r_1}]_\theta
 \subset
 {\mathcal E}^{s}_{pqr}
 \end{equation}
 with equivalence of norms.
 Based on (\ref{eq:170426s-11}),
 we prove (\ref{eq:170426s-12})
 as follows:
 First,
 if $G \in {\mathcal G}({\mathcal E}^{s_0}_{p_0q_0r_0},
 {\mathcal E}^{s_1}_{p_1q_1r_1})$.
 Then
 $$F_j(z) := -i 2^{j}(G(z+2^{-j}i)-G(z))$$
 belongs to ${\mathcal F}({\mathcal E}^{s_0}_{p_0q_0r_0},
 {\mathcal E}^{s_1}_{p_1q_1r_1})$
 and the norm is less than or equal to
 $\|G\|_{{\mathcal G}({\mathcal E}^{s_0}_{p_0q_0r_0},
 	{\mathcal E}^{s_1}_{p_1q_1r_1})}$.
 According to (\ref{eq:170426s-11}), 
 we have
 $$\|F_j\|_{{\mathcal E}^{s}_{pqr}}\lesssim\|G\|_{{\mathcal G}({\mathcal E}^{s_0}_{p_0q_0r_0},
 	{\mathcal E}^{s_1}_{p_1q_1r_1})}.$$
 Since
 $F_j \to G(\theta)$ as $j \to \infty$ in 
 ${\mathcal E}^{s_0}_{p_0q_0r_0}+	{\mathcal E}^{s_1}_{p_1q_1r_1}$,
 and hence in ${\mathcal S}'({\mathbb R}^n)$,
 by the Fatou property 
 $\|G(\theta)\|_{{\mathcal E}^{s}_{pqr}}\lesssim\|G\|_{{\mathcal G}({\mathcal E}^{s_0}_{p_0q_0r_0},
 	{\mathcal E}^{s_1}_{p_1q_1r_1})}$.
 
 Conversely, let $f \in 	{\mathcal E}^{s}_{pqr}$ with norm $1$.
 Define linear functions
 $\rho_1,\rho_2,\rho_3,\rho_4$
 of the variable $z \in {\mathbb C}$ uniquely by
 \begin{align*}
 \rho_1(l):=s\frac{r}{r_l}-s_l, \quad
 \rho_2(l):=\frac{p}{p_l}-\frac{r}{r_l}, \quad
 \rho_3(l):=1-\frac{p}{p_l}, \quad
 \rho_4(l):=\frac{r}{r_l}, \quad l=0,1.
 \end{align*}
 Since
 $\rho_k(\theta)=(1-\theta)\rho_k(0)+\theta\,\rho_k(1)$,
 $\rho_k(\theta)=0$
 for $k=1,2,3$
 and $\rho_4(\theta)=1$.
 Define
 \[
 F_\nu(z):=
 \varphi_\nu(D)\left[
 2^{\nu \rho_1(z)}
 \left(\sum_{j=1}^\nu |2^{j s}\varphi_j(D)|^r
 \right)^{\frac{\rho_2(z)}{r}}
 \| f \|_{{\mathcal E}^{s}_{pqr}}^{\rho_3(z)}
 {\rm sgn}(\varphi_\nu(D)f)
 |\varphi_{\nu}(D)f|^{\rho_4(z)}\right],
 \]
 \[
 F(z):=
 \sum_{\nu=0}^\infty F_\nu(z),
 \]
 and 
 \[
 G(z):=\int_{\theta}^{z}
 F(w) \ dw.
 \]
 In Section \ref{ss331},
we prove
\begin{equation}\label{eq:170426-11}
G \in {\mathcal G}({\mathcal E}_{p_0q_0r_0}^{s_0},
{\mathcal E}_{p_1q_1r_1}^{s_1}).
\end{equation}
 So,
 \[
 \|f\|_{[{\mathcal E}_{p_0q_0r_0}^{s_0},{\mathcal E}_{p_1q_1r_1}^{s_1}]^\theta}
 \le
 \|G'\|_{{\mathcal G}({\mathcal E}_{p_0q_0r_0}^{s_0},{\mathcal E}_{p_1q_1r_1}^{s_1})}
 \lesssim
1.
 \]
\subsection{Proof of Theorem \ref{th-150928-2}}

Suppose that
$\sum_{j=0}^N\varphi_j(D)f \to f$
in  
${\mathcal E}^s_{p q r}$ as $N\to \infty$.
Let $\alpha$ be a multiindex.
Then since
\begin{align*}
\sum_{j=0}^N\varphi_j(D)f
=
\sum_{j=0}^N{\mathcal F}^{-1}
\left[\sum_{k=0}^{N+1}\varphi_k \cdot \varphi_j{\mathcal F}f\right]
=c_n
\sum_{j=0}^N
\sum_{k=0}^{N+1}{\mathcal F}^{-1}\varphi_k*\varphi_j(D)f
\end{align*}
for some constant $c_n>0$,
it follows that
\[
\frac{\partial^\alpha}{\partial x^\alpha}\sum_{j=0}^N\varphi_j(D)f
=c_{n,\alpha}
\sum_{j=0}^N
\sum_{k=0}^{N+1}{\mathcal F}^{-1}[\xi^\alpha\varphi_k]*\varphi_j(D)f.
\]
Since
$
{\mathcal F}^{-1}[\xi^\alpha\varphi_k] \in 
{\mathcal S} \subset L^1
$
and
$p,q,r>1$, we have
\[
\frac{\partial^\alpha}{\partial x^\alpha}\sum_{j=0}^N\varphi_j(D)f
\in \mathcal{E}{}^s_{pqr}.
\]
Thus $f \in \overset{\diamond}{\mathcal E}{}^s_{p q r}$.

Suppose instead that $f \in \overset{\diamond}{\mathcal E}{}^s_{p q r}$.
Let $\varepsilon>0$ be arbitrary.
Then by the definition of $\overset{\diamond}{\mathcal E}{}^s_{p q r}$,
we can find $g \in {\mathcal E}^s_{p q r}$
such that $\partial^\alpha g \in {\mathcal E}^s_{p q r}$
and that
$\|g-f\|_{{\mathcal E}^s_{p q r}}<\varepsilon$.
Then for $N \ge 3$, we have
\begin{align*}
\left\|g-\sum_{j=0}^N\varphi_j(D)g
\right\|_{{\mathcal E}^s_{p q r}}
&\le 2
\left\|\left(
\sum_{k=0}^\infty
2^{k s r}\left|\varphi_k(D)\left[
g-\sum_{j=0}^N\varphi_j(D)g\right]\right|^r
\right)^{\frac1r}
\right\|_{{\mathcal M}^p_q}.
\end{align*}
From the size of the support condition,
we obtain
\begin{align*}
\left\|g-\sum_{j=0}^N\varphi_j(D)g
\right\|_{{\mathcal E}^s_{p q r}}
&\le 2
\left\|\left(
\sum_{k=N}^\infty
2^{k s r}\left|\varphi_k(D)\left[
g-\sum_{j=0}^N\varphi_j(D)g\right]\right|^r
\right)^{\frac1r}
\right\|_{{\mathcal M}^p_q}.
\end{align*}
Since ${\mathcal F}^{-1}\varphi_k(\xi)=2^{k n}{\mathcal F}^{-1}\varphi(2^{-k}\xi)$,
we can use the Hardy-Littlewood maximal operator to have
\begin{align*}
\left\|g-\sum_{j=0}^N\varphi_j(D)g
\right\|_{{\mathcal E}^s_{p q r}}
&\le C
\left\|\left(
\sum_{k=N}^\infty
2^{k s r}M[\varphi_k(D)g]^r
\right)^{\frac1r}
\right\|_{{\mathcal M}^p_q}.
\end{align*}
By Theorem \ref{thm:vec-maxi},
we have
\begin{align}\label{eq:170427-11}
\left\|g-\sum_{j=0}^N\varphi_j(D)g
\right\|_{{\mathcal E}^s_{p q r}}
&\le C
\left\|\left(
\sum_{k=N}^\infty
2^{k s r}|\varphi_k(D)g|^r
\right)^{\frac1r}
\right\|_{{\mathcal M}^p_q}.
\end{align}
Let us set
\[
\varphi^*_k(\xi)
:=
\frac{\varphi_{k-1}(\xi)+\varphi_k(\xi)+\varphi_{k+1}(\xi)}{|2^{-k}\xi|^2}.
\]
Then we have
$\varphi_k(D)g=-2^{-2k}\varphi_k^*(D)\varphi_k(D)[\Delta g].$
Inserting this relation into (\ref{eq:170427-11}),
we obtain
\begin{align*}
\left\|g-\sum_{j=0}^N\varphi_j(D)g
\right\|_{{\mathcal E}^s_{p q r}}
&\le C
\left\|\left(
\sum_{k=N}^\infty
2^{k (s-2) r}
|\varphi^*_k(D)\varphi_k(D)[\Delta g]|^r
\right)^{\frac1r}
\right\|_{{\mathcal M}^p_q}.
\end{align*}
Again by using the convolution and the maximal operator,
we obtain
\begin{align*}
\left\|g-\sum_{j=0}^N\varphi_j(D)g
\right\|_{{\mathcal E}^s_{p q r}}
&\le C2^{-2N}
\left\|\left(
\sum_{k=N}^\infty
2^{k s r}M\left[\varphi_k(D)[\Delta g]\right]^r
\right)^{\frac1r}
\right\|_{{\mathcal M}^p_q}.
\end{align*}
Using Theorem \ref{thm:vec-maxi} once more,
we have
\begin{align*}
\left\|g-\sum_{j=0}^N\varphi_j(D)g
\right\|_{{\mathcal E}^s_{p q r}}
&\le C2^{-2N}
\left\|\left(
\sum_{k=N}^\infty
2^{k s r}|\varphi_k(D)[\Delta g]|^r
\right)^{\frac1r}
\right\|_{{\mathcal M}^p_q}\\ 
&\le 
C2^{-2N}\|\Delta g\|_{{\mathcal E}^s_{p q r}}.
\end{align*}
Since $\Delta g \in {\mathcal E}^s_{p q r}$,
there exists $N_0 \in {\mathbb N}$ such that
\[
\left\|g-\sum_{j=0}^N\varphi_j(D)g
\right\|_{{\mathcal E}^s_{p q r}}
<\varepsilon
\]
as long as $N \ge N_0$.

If we use Theorem \ref{thm:vec-maxi},
we obtain
\[
\left\|
\sum_{j=0}^N\varphi_j(D)g
-
\sum_{j=0}^N\varphi_j(D)f
\right\|_{{\mathcal E}^s_{p q r}}
\le C
\|f-g\|_{{\mathcal E}^s_{p q r}}
\le C\varepsilon.
\]
Thus, if $N \ge N_0$, then we have
\begin{align*}
\lefteqn{
\left\|f-
\sum_{j=0}^N\varphi_j(D)f
\right\|_{{\mathcal E}^s_{p q r}}
}\\
&\le
\|f-g\|_{{\mathcal E}^s_{p q r}}
+
\left\|g-
\sum_{j=0}^N\varphi_j(D)g
\right\|_{{\mathcal E}^s_{p q r}}
+
\left\|
\sum_{j=0}^N\varphi_j(D)g
-
\sum_{j=0}^N\varphi_j(D)f
\right\|_{{\mathcal E}^s_{p q r}}\\
&\le (2+C)\varepsilon,
\end{align*}
as required.

\subsection{Proof of $\subset$ in $(\ref{eq:151113-12})$}

Let 
$f\in [
\overset{\diamond}{\mathcal E}{}^{s_0}_{p_0q_0r_0},
\overset{\diamond}{\mathcal E}{}^{s_1}_{p_1q_1r_1}]^\theta$.
By Lemma \ref{lem:dense2} and Proposition \ref{cr-151002-1}, we have
\[f \in 
{\mathcal E}^s_{p q r}
\cap
\overline{\overset{\diamond}{{\mathcal E}}{}^s_{pqr}}
^{{\mathcal E}^{s_0}_{p_0q_0r_0}+{\mathcal E}^{s_1}_{p_1q_1r_1}}.
\]
Therefore,
\[
f=f_k+f_{k,0}+f_{k,1},
\]
where 
$f_k \in \overset{\diamond}{{\mathcal E}}{}^s_{pqr}$,
$f_{k,0} \in {\mathcal E}^{s_0}_{p_0q_0r_0}$,
$f_{k,1} \in {\mathcal E}^{s_1}_{p_1q_1r_1}$
for each $k \in {\mathbb N}$
and
\[
\|f_{k,0}\|_{{\mathcal E}^{s_0}_{p_0q_0r_0}}
+
\|f_{k,1}\|_{{\mathcal E}^{s_1}_{p_1q_1r_1}} \le k^{-1}.
\]
For $0<a<1$ and $b>0$, we see that
\begin{align*}
S(f;a,J,r,s)
&=
\chi_{[0,b S(f;r,s)]}
\left(
\sqrt[r]{\sum_{j=J}^\infty |2^{j s}\varphi_j(D)f|^r}
\right) S(f;a,J,r,s)\\
&\quad+
\chi_{(b S(f;r,s),\infty)}
\left(
\sqrt[r]{\sum_{j=J}^\infty |2^{j s}\varphi_j(D)f|^r}
\right) S(f;a,J,r,s)\\
&\le b S(f;r,s)
+
\chi_{(a b,a^{-1}]}
\left(
\sqrt[r]{\sum_{j=J}^\infty |2^{j s}\varphi_j(D)f|^r}
\right)
\sqrt[r]{\sum_{j=J}^\infty |2^{j s}\varphi_j(D)f|^r}.
\end{align*}
Thus,
\begin{align*}
\lefteqn{
\limsup_{J \to \infty}\|S(f;a,J,r,s)\|_{{\mathcal M}^p_q}
}\\
&\lesssim
b\|f\|_{{\mathcal E}^s_{p q r}}
+
\limsup_{J \to \infty}
\left\|
\chi_{(a b,a^{-1}]}
\left(
\sqrt[r]{\sum_{j=J}^\infty |2^{j s}\varphi_j(D)f|^r}
\right)
\sqrt[r]{\sum_{j=J}^\infty |2^{j s}\varphi_j(D)f|^r}
\right\|_{{\mathcal M}^p_q}.
\end{align*}
Once we show that
\begin{equation}\label{eq:s151217-1}
\lim_{J \to \infty}
\left\|
\chi_{(a b,a^{-1}]}
\left(
\sqrt[r]{\sum_{j=J}^\infty |2^{j s}\varphi_j(D)f|^r}
\right)
\sqrt[r]{\sum_{j=J}^\infty |2^{j s}\varphi_j(D)f|^r}
\right\|_{{\mathcal M}^p_q}=0,
\end{equation}
then we have the desired result.
By setting
\[
\Phi_A(t)
:=\max(0,(t-A)(A^{-1}-t))^r \quad (t \in {\mathbb R}),
\]
where $A>0$, we have only to show that
\begin{equation}
\lim_{J \to \infty}
\left\|
\Phi_A\left(\sqrt[r]{\sum_{j=J}^\infty |2^{j s}\varphi_j(D)f|^r}\right)
\right\|_{{\mathcal M}^p_q}
=0
\end{equation}
in ${\mathcal M}^p_q$ for all $A>0$.

By the mean-value theorem, we have
\begin{align*}
\lefteqn{
\left|
\Phi_A\left(\sqrt[r]{\sum_{j=J}^\infty |2^{j s}\varphi_j(D)f|^r}\right)
-
\Phi_A\left(\sqrt[r]{\sum_{j=J}^\infty |2^{j s}\varphi_j(D)f_k|^r}\right)
\right|
}\\
&\lesssim_A
\min\left(1,\left|
\sqrt[r]{\sum_{j=J}^\infty |2^{j s}\varphi_j(D)f|^r}
-
\sqrt[r]{\sum_{j=J}^\infty |2^{j s}\varphi_j(D)f_k|^r}
\right|\right)
\\
&\le
\min\left(1,
\sqrt[r]{\sum_{j=J}^\infty |2^{j s}\varphi_j(D)(f-f_k)|^r}\right)
=
\min\left(1,
\sqrt[r]{\sum_{j=J}^\infty |2^{j s}\varphi_j(D)(f_{k,0}+f_{k,1})|^r}\right).
\end{align*}
We let
\begin{align*}
{\mathfrak A}
&:=
\chi_{[A,A^{-1}]}(S(f;r,s))
\Phi_A\left(\sqrt[r]{\sum_{j=J}^\infty |2^{j s}\varphi_j(D)f|^r}\right)\\
{\mathfrak B}
&:=
\chi_{[A,A^{-1}]}(S(f;r,s))
\Phi_A\left(\sqrt[r]{\sum_{j=J}^\infty |2^{j s}\varphi_j(D)f_k|^r}\right).
\end{align*}
So, we have
\begin{align}\label{eq:17422-1} 
\|{\mathfrak A}-{\mathfrak B}\|_{{\mathcal M}^p_q}
&\lesssim_A
\left\|
\chi_{[A,A^{-1}]}(S(f;r,s))
\min\left\{
\left(
\sum_{j=J}^\infty
|2^{j s}\varphi_j(D)f_{k,0}|^r
\right)^{\frac1r},1
\right\}
\right\|_{{\mathcal M}^{p}_{q}}
\nonumber
\\
&\quad+
\left\|
\chi_{[A,A^{-1}]}(S(f;r,s))
\min\left\{
\left(
\sum_{j=J}^\infty
|2^{j s}\varphi_j(D)f_{k,1}|^r
\right)^{\frac1r},1\right\}
\right\|_{{\mathcal M}^{p}_{q}}.
\end{align}
Recall that we are assuming $r_0=r_1=r$ and $s_0=s_1=s$.
By using $q_0>q>q_1$,
$\frac{p}{q}=\frac{p_0}{q_0}
=\frac{p_1}{q_1}$, and
the H\"{o}lder inequality, we get
\begin{align}\label{eq:17422-2}
\lefteqn{
\|{\mathfrak A}-{\mathfrak B}\|_{{\mathcal M}^p_q}
}\nonumber\\
&\lesssim_A
\left\|
\left(
\sum_{j=0}^\infty
|2^{j s_0}\varphi_j(D)f_{k,0}|^{r_0}
\right)^{\frac1{r_0}}
\right\|_{\cM^{p_0}_{q_0}}
\left\|
\left(
\sum_{j=0}^\infty
|2^{j s}\varphi_j(D)f|^r
\right)^{\frac1r}
\right\|_{{\mathcal M}^{p}_{q}}^{1-\frac{q}{q_0}}
\nonumber
\\
&\quad+
\left\|
\left(
\sum_{j=0}^\infty
|2^{j s_1}\varphi_j(D)f_{k,1}|^{r_1}
\right)^{\frac1{r_1}}
\right\|_{ \cM^{p_1}_{q_1}}^{\frac{q_1}{q}}
\nonumber
\\
&\lesssim_A
\|f_{k,0}\|_{{\mathcal E}^{s_0}_{p_0q_0r_0}}
\|f\|_{{\mathcal E}^{s}_{pqr}}^{1-\frac{q}{q_0}}
+
\|f_{k,1}\|_{{\mathcal E}^{s_1}_{p_1q_1r_1}}
^{\frac{q_1}{q}}.
\end{align}
Consequently,
\begin{eqnarray*}
&&\left\|
\chi_{[A,A^{-1}]}(S(f;r,s))
\Phi_A\left(\sqrt[r]{\sum_{j=J}^\infty |2^{j s}\varphi_j(D)f|^r}\right)
\right\|_{\cM^p_q}\\
&& \lesssim_A
\|f_{k,0}\|_{{\mathcal E}^{s_0}_{p_0q_0r_0}}
\|f\|_{{\mathcal E}^{s}_{pqr}}^{1-\frac{q}{q_0}}
+
\|f_{k,1}\|_{{\mathcal E}^{s_1}_{p_1q_1r_1}}^{\frac{q_1}{q}}
+\left\|
\sqrt[r]{\sum_{j=J}^\infty |2^{j s}\varphi_j(D)f_k|^r}
\right\|_{{\mathcal M}^p_q}.
\end{eqnarray*}
By letting $J \to \infty$,
we obtain
\begin{align*}
\limsup_{J \to \infty}
\left\|
\Phi_A\left(\sqrt[r]{\sum_{j=J}^\infty |2^{j s}\varphi_j(D)f|^r}\right)
\right\|_{{\mathcal M}^p_q}
\lesssim_A
\|f_{k,0}\|_{{\mathcal E}^{s_0}_{p_0q_0r_0}}
\|f\|_{{\mathcal E}^{s}_{pqr}}^{1-\frac{q}{q_0}}
+
\|f_{k,1}\|_{{\mathcal E}^{s_1}_{p_1q_1r_1}}
^{\frac{q_1}{q}}.
\end{align*}
Finally letting $k \to \infty$,
we obtain (\ref{eq:s151217-1}).

\subsection{Proof of $\subset$ in $(\ref{eq:151113-11})$}

We readily obtain
the inclusion by combining 
\eqref{eq:151113-12}, 
\eqref{eq:Bergh}, 
Proposition \ref{cr-151002-1}
and 
$$
[
\overset{\diamond}{{\mathcal E}}{}^{s_0}_{p_0q_0r_0},
\overset{\diamond}{{\mathcal E}}{}^{s_1}_{p_1q_1r_1}]_\theta
\subset
[
{\mathcal E}^{s_0}_{p_0q_0r_0},
{\mathcal E}^{s_1}_{p_1q_1r_1}]_\theta.
$$

\subsection{Proof of $\supset$ in $(\ref{eq:151113-11})$}

We choose $\psi$ 
so that
\[
{\rm supp}(\psi) \subset B(8)
\]
and define
\[
\varphi_j:=\psi(2^{-j}\cdot)-\psi(2^{-j+1}\cdot).
\]
Write $\varphi_j=\varphi(2^{-j}\cdot)$
as before.

Let 
$f \in
\overset{\diamond}{{\mathcal E}}{}^{s}_{pqr}
\cap
[
{\mathcal E}^{s_0}_{p_0q_0r_0},
{\mathcal E}^{s_1}_{p_1q_1r_1}]_\theta.$
Then since
$f \in
\overset{\diamond}{{\mathcal E}}{}^{s}_{pqr}$,
we have
\[
f=\psi(D)f+\sum_{j=1}^\infty \varphi_j(D)f
\]
in ${\mathcal E}^s_{pqr}$.
We write
$\displaystyle
f_J:=\psi(D)f+\sum_{j=1}^J \varphi_j(D)f.
$
By virtue of Theorem \ref{thm:170426s-2}
and
$f_J-f_{J'} \in 
[
{\mathcal E}^{s_0}_{p_0q_0r_0},
{\mathcal E}^{s_1}_{p_1q_1r_1}]_\theta$
for any $J,J' \in {\mathbb N}$,
we can use
(\ref{eq:170426s-12}) and (\ref{eq:Bergh})
to have
\[
\|f_J-f_{J'}\|_{
[
{\mathcal E}^{s_0}_{p_0q_0r_0},
{\mathcal E}^{s_1}_{p_1q_1r_1}]_\theta}
=
\|f_J-f_{J'}\|_{
[
{\mathcal E}^{s_0}_{p_0q_0r_0},
{\mathcal E}^{s_1}_{p_1q_1r_1}]^\theta}
\sim
\|f_J-f_{J'}\|_{{\mathcal E}^{s}_{pqr}}.
\]
Since
$f_J-f_{J'}
\in [
{\mathcal E}^{s_0}_{p_0q_0r_0},
{\mathcal E}^{s_1}_{p_1q_1r_1}]_\theta
$
and
${\rm supp}({\mathcal F}(f_J-f_{J'}))$
is a compact set in ${\mathbb R}^n \setminus \{0\}$,
we can find $F_{J,J'} \in {\mathcal F}(
\overset{\diamond}{\mathcal E}{}^{s_0}_{p_0q_0r_0},
\overset{\diamond}{\mathcal E}{}^{s_1}_{p_1q_1r_1})$
such that
\[
F_{J,J'}(\theta)=f_J-f_{J'}, \quad
\|F_{J,J'}\|_{ 
\cF({\mathcal E}^{s_0}_{p_0q_0r_0},
{\mathcal E}^{s_1}_{p_1q_1r_1})}
\lesssim
\|f_J-f_{J'}\|_{ [
	{\mathcal E}^{s_0}_{p_0q_0r_0},
	{\mathcal E}^{s_1}_{p_1q_1r_1}]_\theta}
\]
Thus, it follows that
\begin{align*}
\|f_J-f_{J'}\|_{ [
\overset{\diamond}{\mathcal E}{}^{s_0}_{p_0q_0r_0},
\overset{\diamond}{\mathcal E}{}^{s_1}_{p_1q_1r_1}]_\theta}
&\le
\|F_{J,J'}\|_{{\mathcal F}(
\overset{\diamond}{\mathcal E}{}^{s_0}_{p_0q_0r_0},
\overset{\diamond}{\mathcal E}{}^{s_1}_{p_1q_1r_1})}\\
&\lesssim
\|F_{J,J'}(\theta)\|_{ [
{\mathcal E}^{s_0}_{p_0q_0r_0},
{\mathcal E}^{s_1}_{p_1q_1r_1}]_\theta}\\
&\lesssim
\|f_J-f_{J'}\|_{{\mathcal E}^{s}_{pqr}}.
\end{align*}
Here we used \cite[Corollary 1.11]{YYZ13}
for the last inequality.
Hence
$\{f_J\}_{J=1}^\infty$
is a Cauchy sequence in 
$[
\overset{\diamond}{\mathcal E}{}^{s_0}_{p_0q_0r_0},
\overset{\diamond}{\mathcal E}{}^{s_1}_{p_1q_1r_1}]_\theta$.
Since
$\{f_J\}_{J=1}^\infty$
converges to $f$ in 
${\mathcal E}{}^{s_0}_{p_0q_0r_0}
+
{\mathcal E}{}^{s_1}_{p_1q_1r_1}.$
We see that
$\{f_J\}_{J=1}^\infty$
converges to $f$ in 
$[
\overset{\diamond}{\mathcal E}{}^{s_0}_{p_0q_0r_0},
\overset{\diamond}{\mathcal E}{}^{s_1}_{p_1q_1r_1}]_\theta$.

\subsection{Proof of $\supset$ in $(\ref{eq:151113-12})$}
Let $f \in {\mathcal E}^s_{p q r}$ be such that
$\displaystyle
\lim_{J \to \infty}
\|S(f;a,J,r,s)\|_{{\mathcal M}^p_q}
=0
$
for all $0<a<1$.
We suppose that $f$ has ${\mathcal E}^s_{p q r}$-norm $1$.
Choose
$\varphi \in {\mathcal S}$ 
so that $\varphi \ge 0$ and
$
\chi_{B(2)} \le \varphi^2 \le \chi_{B(3)}.
$
Write $\varphi_0:=\varphi$ and  $\varphi_j:=\sqrt{\varphi_0(2^{-j}\cdot)^2-\varphi_0(2^{-j+1} \cdot)^2}$  for $j\in \N$.
Then, $\{\varphi_j\}_{j=0}^\infty$ satisfies
\begin{equation*}
\sum_{j=0}^\infty\varphi_j^2=1.
\end{equation*}
For each $\nu \in \N \cup \{ 0\}$, define 
\[
V_\nu(f):=
\left(\sum_{j=0}^\nu |2^{js}\varphi_j(D)f|^r
\right)^{\frac{1}{r}}
\]
For $z\in \overline{U}$, we define 
\begin{equation}\label{eq:170301-52}
F(z)
:=
\sum_{\nu=0}^\infty
\varphi_\nu(D) 
\left(
V_{\nu}(f)^{p\left( \frac{1-z}{p_0}+\frac{z}{p_1} \right)-1}
\cdot 
\varphi_{\nu}(D)f
\right)
\end{equation}
and
\[
G(z):=\int_\theta^z F(w)\,dw.
\]
We prove in Section \ref{ss331} that 
\begin{align}\label{eq:17414-1}
G \in {\mathcal G}(
\overset{\diamond}{{\mathcal E}}{}^{s_0}_{p_0q_0r_0},
\overset{\diamond}{{\mathcal E}}{}^{s_1}_{p_1q_1r_1}),
\quad
\|G\|_{{\mathcal G}(
\overset{\diamond}{{\mathcal E}}{}^{s_0}_{p_0q_0r_0},
\overset{\diamond}{{\mathcal E}}{}^{s_1}_{p_1q_1r_1})} \lesssim 1.
\end{align}
From \eqref{eq:17414-1} and $f=G'(\theta)$, we conclude that 
$f\in [
\overset{\diamond}{{\mathcal E}}{}^{s_0}_{p_0q_0r_0},
\overset{\diamond}{{\mathcal E}}{}^{s_1}_{p_1q_1r_1}]^\theta$,
as desired.

\section{Proof of \eqref{eq:170426-11} and \eqref{eq:17414-1}}\label{ss331}

Let $p_0,p_1,p,\ldots$ be the same as before.
We check the conditions of membership of $\cG(
\overset{\diamond}{{\mathcal E}}{}^{s_0}_{p_0q_0r_0},
\overset{\diamond}{{\mathcal E}}{}^{s_1}_{p_1q_1r_1})$  by proving the following lemmas.
\begin{lemma}\label{lem:17415-1}\
\begin{enumerate}
\item
Let $f \in  
{\mathcal E}^{s}_{pqr}$.
For $z\in \overline{U}$, we have $G(z)\in 
{\mathcal E}^{s_0}_{p_0q_0r_0}
+
{\mathcal E}^{s_1}_{p_1q_1r_1}$. 
Moreover,
\begin{align}\label{eq:170323-2a}
\sup\limits_{z\in \overline{U}}\left\|
\frac{G(z)}{1+|z|}
\right\|_{{\mathcal E}^{s_0}_{p_0q_0r_0}
	+{\mathcal E}^{s_1}_{p_1q_1r_1}}<\infty.
\end{align}
\item
Let $f \in  
\overset{\diamond}{{\mathcal E}}{}^{s}_{pqr}$.
For $z\in \overline{U}$, we have $G(z)\in 
\overset{\diamond}{{\mathcal E}}{}^{s_0}_{p_0q_0r_0}
+
\overset{\diamond}{{\mathcal E}}{}^{s_1}_{p_1q_1r_1}$. 
Moreover,
\begin{align}\label{eq:170323-2}
\sup\limits_{z\in \overline{U}}\left\|
\frac{G(z)}{1+|z|}
\right\|_{\overset{\diamond}{{\mathcal E}}{}^{s_0}_{p_0q_0r_0}
	+
	\overset{\diamond}{{\mathcal E}}{}^{s_1}_{p_1q_1r_1}}<\infty.
\end{align}
\end{enumerate}
\end{lemma}
\begin{proof}
We concentrate on (\ref{eq:170323-2});
the proof of (\ref{eq:170323-2a}) being simpler.
For each $z\in \overline{U}$, we define
\begin{align*}
F_0(z)\
:=
\sum_{\nu=0}^{\infty}
\varphi_\nu(D)
\left(
V_\nu(f)^{p \left(\frac{1-z}{p_0}+\frac{z}{p_1} \right)-1}
\cdot \varphi_\nu(D)f \cdot 
 \chi_{\{V_\nu(f)\le 1\}}
\right),
\end{align*}
\[
\
F_1(z):=F(z)-F_0(z), \ 
G_0(z):= \int_{\theta}^{z} F_0(w) \ dw, \ 
{\rm and}\ 
G_1(z):=\int_{\theta}^{z} F_1(w) \ dw.
\]
We shall  show that
\begin{align}\label{eq:161016-1}
G_0(z)\in \overset{\diamond}{\mathcal{E}}{}^{s_0}_{p_0q_0r_0}
\end{align}
and 
\begin{align}\label{eq:161016-2}
G_1(z)\in \overset{\diamond}{\mathcal{E}}{}^{s_1}_{p_1q_1r_1}.
\end{align}
Let $J\in \mathbb{N} \cap [5,\infty)$.
We use \eqref{eq:150928-71}, \eqref{eq:161018-1}, and the fact that $\varphi_l \varphi_j=0$ whenever $|l-j|\ge 2$
to obtain
\begin{align}\label{eq:17418-1}
\left\|
\sum_{\ell=J}^{\infty}
\varphi_\ell(D)(G_0(z))
\right\|_{\mathcal{E}^{s_0}_{p_0q_0r_0}}
&=
\left\|
\left(
\sum_{j=J-1}^{\infty}
\left|\varphi_j(D) \left[\sum_{\ell=J}^{\infty}
\varphi_\ell(D)(2^{js_0} G_0(z)) \right] \right|^{r_0}
\right)^{\frac{1}{r_0}}
\right\|_{\cM^{p_0}_{q_0}}
\nonumber
\\
&\lesssim
\left\|
\left(
\sum_{j=J-1}^{\infty}
|\varphi_j(D)(2^{js_0} G_0(z))|^{r_0}
\right)^{\frac{1}{r_0}}
\right\|_{\cM^{p_0}_{q_0}}.
\end{align}
Let $Q:=\frac{p}{p_1}-\frac{p}{p_0}$.
Combining 
\begin{align*}
&\sum_{j=J-1}^{\infty}
|\varphi_{j}(D)[ 2^{js_0} G_0(z)]|^{r_0}
\\
&\lesssim
\sum_{j=J-1}^{\infty}
2^{js_0r_0} 
M\left(\varphi_j(D)
\left[\chi_{\{ V_j(f) \le 1 \}}  \varphi_{j}(D)f \cdot V_j(f)^{\frac{p}{p_0}-1}  \int_{\theta}^{z} V_j(f)^{Qw} \ dw \right]  \right )^{r_0},
\end{align*} 
\eqref{eq:150928-71}, \eqref{eq:161018-1}, and 
\eqref{eq:17418-1}, we get
\begin{align}\label{eq:17423-1}
&\left\|
\sum_{\ell=J}^{\infty}
\varphi_\ell(D)(G_0(z))
\right\|_{\mathcal{E}^{s_0}_{p_0q_0r_0}}
\lesssim
\|I_1\|_{\cM^{p_0}_{q_0}}
+
\|I_2\|_{\cM^{p_0}_{q_0}} 
\end{align}
where
\begin{align*}
\lefteqn{
I_1
}\\
&:=
\chi_{[a,a^{-1}]}(S(f;r,s)) 
\left(
\sum_{j=J-1}^\infty
|2^{js} \varphi_j(D)f|^{r}
V_j(f)^{\frac{pr}{p_0}-r}
\left|
\chi_{\{V_j(f)\le 1  \}}
\int_{\theta}^{z}
V_j(f)^{Qw} \ dw 
\right|^r
\right)^{\frac{1}{r}}
\end{align*}
and 
\begin{align*}
I_2
&:=
(1-\chi_{[a,a^{-1}]}(S(f;r,s)))
\\
&\times 
\left(
\sum_{j=J-1}^\infty
|2^{js} \varphi_j(D)f|^{r}
V_j(f)^{\frac{pr}{p_0}-r}
\left|
\frac{V_j(f)^{Qz}-V_j(f)^{Q\theta}}{\log(V_j(f))}
\chi_{\{V_j(f) \le 1  \}}
\right|^{r}
\right)^{\frac{1}{r}}.
\end{align*}
By virtue of Lemma \ref{lem:151109-1}, we get
\begin{align}\label{eq:170403-2}
\|I_1\|_{\cM^{p_0}_{q_0}}
&\lesssim
(1+|z|)
\left\|
\chi_{[a,a^{-1}]}(S(f;r,s)) \left(
\sum_{j=J-1}^\infty
|2^{js} \varphi_j(D)f|^{r}
V_j(f)^{\frac{pr}{p_0}-r}
\right)^{\frac{1}{r}}
\right\|_{\cM^{p_0}_{q_0}}
\nonumber
\\
&\lesssim
(1+|z|)
\left\|
\chi_{[a,a^{-1}]}(S(f;r,s))
\left(
\sum_{\ell=J-2}^{\infty}
|2^{\ell s} \varphi_\ell(D)f|^{r}
\right)^{\frac{p}{rp_0}}
\right\|_{\cM^{p_0}_{q_0}}
\nonumber
\\
&\le 
(1+|z|) 
\|S(f;a,J-1,r,s)\|_{\cM^p_q}^{\frac{p}{p_0}}.
\end{align}
We combine  Lemmas \ref{lem:161109-1},
\ref{lem:151112-1},  and \ref{lem:161109-2} to obtain
\begin{align*}
I_2 &\lesssim
(1-\chi_{[a,a^{-1}]}(S(f;r,s)) )
\left(
\sum_{j=0}^{\infty}
|2^{js} \varphi_j(D)f|^{r}
\Phi_{\frac{p}{p_0}}(V_j(f))^{r}
\right)^{\frac{1}{r}}
\\
&\lesssim
(1-\chi_{[a,a^{-1}]}(S(f;r,s)) )
\left(
\Psi_{\frac{p}{p_0}}
\left(
S(f;r,s)^r
\right)^{\frac{1}{r}}
\right)
\\
&\lesssim
\left(
a^{\frac{p}{p_0}(r-1)}
+
\left(
\log\left(
\sqrt{a}+\frac{1}{\sqrt{a}}
\right)
\right)^{-r}
\right)
S(f;r,s)^{\frac{p}{p_0}}.
\end{align*}
Consequently,
\begin{align}\label{eq:170403-3}
\|I_2\|_{\cM^{p_0}_{q_0}}
\lesssim
\left(
a^{\frac{p}{p_0}(r-1)}
+
\left(
\log\left(
\sqrt{a}+\frac{1}{\sqrt{a}}
\right)
\right)^{-r}
\right)
\|f\|_{\mathcal{E}^s_{pqr}}^{\frac{p}{p_0}}
\end{align}
Therefore, by combing \eqref{eq:17423-1}, \eqref{eq:170403-2} and \eqref{eq:170403-3},  and then taking $J\to \infty$ and $a\to 0^+$, 
we have
\[
\lim\limits_{J\to \infty}
\left\|
\sum_{\ell=J}^{\infty}
\varphi_\ell(D)(G_0(z))
\right\|_{\mathcal{E}^{s_0}_{p_0q_0r_0}}
=0.
\]
Thus, $G_0(z) \in \overset{\diamond}{\mathcal{E}}{}^{s_0}_{p_0q_0r_0}$.
By a similar argument, we also have
\[
\lim\limits_{J\to \infty}
\left\|
\sum_{\ell=J}^{\infty}
\varphi_\ell(D)(G_1(z))
\right\|_{\mathcal{E}^{s_1}_{p_1q_1r_1}}
=0,
\]
which implies \eqref{eq:161016-2}. Since $G(z)=G_0(z)+G_1(z)$, we have 
$G(z) \in \overset{\diamond}{\mathcal{E}}{}^{s_0}_{p_0q_0r_0}+\overset{\diamond}{\mathcal{E}}{}^{s_1}_{p_1q_1r_1}$,

The proof of \eqref{eq:170323-2} goes as follows. 
By virtue of  \eqref{eq:161018-1} and Theorem \ref{thm:vec-maxi}, we have
\begin{align}\label{eq:170419-2}
&\|G_0(z)\|_{\overset{\diamond}{\mathcal{E}}{}^{s_0}_{p_0q_0r_0}}
\nonumber
\\
&=
\|G_0(z)\|_{{\mathcal{E}}{}^{s_0}_{p_0q_0r_0}}
\nonumber
\\
&\lesssim
\left\|
\varphi_0(D)f \cdot V_0(f)^{\frac{p}{p_0}-1}
 \int_{\theta}^{z} V_0(f)^{Qw} \ dw \chi_{\{V_0(f) \le 1  \}}\right\|_{\cM^{p_0}_{q_0}}
\nonumber
\\
&\quad +
\left\|
\varphi_1(D)f \cdot V_1(f)^{\frac{p}{p_0}-1}
 \int_{\theta}^{z} V_1(f)^{Qw}\ dw \chi_{\{V_1(f) \le 1  \}}
 \right\|_{\cM^{p_0}_{q_0}}
\nonumber
\\
&\quad +
\left\|
\left(
\sum_{l=1}^{\infty}
\left|2^{ls} 
\varphi_l(D)f \cdot V_l(f)^{\frac{p}{p_0}-1}
\int_{\theta}^{z} V_l(f)^{Qw} \ dw \chi_{\{V_l(f) \le 1  \}}\right|^{r_0}
\right)^{\frac{1}{r_0}}
\right\|_{\cM^{p_0}_{q_0}}.
\end{align}
Combining \eqref{eq:170419-2} and
Lemma \ref{lem:151109-1}, 
we have 
\begin{align}\label{eq:17420-2}
\|G_0(z)\|_{\overset{\diamond}{\mathcal{E}}{}^{s_0}_{p_0q_0r_0}}
&\lesssim
(1+|z|)
\|\varphi(D)f\|_{\cM^p_q}^{\frac{p}{p_0}}
+
(1+|z|)
\|\varphi_1(D)f\|_{\cM^p_q}^{\frac{p}{p_0}}
\nonumber
\\
&\quad +
(1+|z|)
\left\|
\left(
\sum_{l=1}^{\infty}
|2^{ls} \varphi_l(D)f|^r
\right)^{\frac{1}{r}}
\right\|_{\cM^p_q}^{\frac{p}{p_0}}
\nonumber
\\
&\quad \lesssim
(1+|z|) \|f\|_{\mathcal{E}{}^{s}_{pqr}}^{\frac{p}{p_0}}.
\end{align}
By a similar argument 
\begin{align}\label{eq:17420-3}
\|G_0(z)\|_{\overset{\diamond}{\mathcal{E}}{}^{s_0}_{p_0q_0r_0}}
\lesssim
(1+|z|) \|f\|_{\mathcal{E}{}^{s}_{pqr}}^{\frac{p}{p_1}}.
\end{align}
Thus, \eqref{eq:170323-2} follows from \eqref{eq:17420-2} and \eqref{eq:17420-3}. 
\end{proof}

\begin{lemma}\label{lem:17415-2}\
\begin{enumerate}
\item
Let $f \in  
{\mathcal E}^{s}_{pqr}$.
Then
the function $G: \overline{U} \to {\mathcal E}^{s_0}_{p_0q_0r_0}
+{\mathcal E}^{s_1}_{p_1q_1r_1}$ is continuous 
and $G: U \to {\mathcal E}^{s_0}_{p_0q_0r_0}
+{\mathcal E}^{s_1}_{p_1q_1r_1} $ is holomorphic.
\item
Let $f \in  \overset{\diamond}{{\mathcal E}}{}^{s}_{pqr}$.
Then
the function $G: \overline{U} \to \overset{\diamond}{{\mathcal E}}{}^{s_0}_{p_0q_0r_0}
+
\overset{\diamond}{{\mathcal E}}{}^{s_1}_{p_1q_1r_1}$ is continuous 
and $G: U \to \overset{\diamond}{{\mathcal E}}{}^{s_0}_{p_0q_0r_0}
+
\overset{\diamond}{{\mathcal E}}{}^{s_1}_{p_1q_1r_1} $ is holomorphic.
\end{enumerate}
\end{lemma}

\begin{proof}
We suppose
 $f \in  \overset{\diamond}{{\mathcal E}}{}^{s}_{pqr}$.
The case of
$f \in  
{\mathcal E}^{s}_{pqr}$
can be handled similarly.
Let $z_1, z_2 \in \overline{U}$. 
By virtue of  \eqref{eq:161018-1} and Theorem \ref{thm:vec-maxi}, we have
\begin{align}\label{eq:170419-4}
\|&G_0(z_1)-G_0(z_2)\|_{\overset{\diamond}{\mathcal{E}}{}^{s_0}_{p_0q_0r_0}}
\nonumber
\\
&=
\|G_0(z_1)- G_0(z_2)\|_{{\mathcal{E}}{}^{s_0}_{p_0q_0r_0}}
\nonumber
\\
&\lesssim
\left\|
\varphi_0(D)f \cdot V_0(f)^{\frac{p}{p_0}-1}
\int_{z_2}^{z_1} V_0(f)^{Qw} \ dw \chi_{\{V_0(f) \le 1  \}}\right\|_{\cM^{p_0}_{q_0}}
\nonumber
\\
&\quad +
\left\|
\varphi_1(D)f \cdot V_1(f)^{\frac{p}{p_0}-1}
\int_{z_2}^{z_1} V_1(f)^{Qw}\ dw \chi_{\{V_1(f) \le 1  \}}
\right\|_{\cM^{p_0}_{q_0}}
\nonumber
\\
&\quad +
\left\|
\left(
\sum_{l=1}^{\infty}
\left|2^{ls} 
\varphi_l(D)f \cdot V_l(f)^{\frac{p}{p_0}-1}
\int_{z_2}^{z_1} V_l(f)^{Qw} \ dw \chi_{\{V_l(f) \le 1  \}}\right|^{r_0}
\right)^{\frac{1}{r_0}}
\right\|_{\cM^{p_0}_{q_0}}.
\end{align}
Combining  \eqref{eq:170419-4}
and Lemma \ref{lem:151109-1},  
we get
\begin{align}\label{eq:17421-1}
\|G_0(z_1)- G_0(z_2)\|_{\overset{\diamond}{\mathcal{E}}{}^{s_0}_{p_0q_0r_0}}
&\lesssim
|z_1-z_2|
\|\varphi(D)f\|_{\cM^p_q}^{\frac{p}{p_0}}
+
|z_1-z_2|
\|\varphi_1(D)f\|_{\cM^p_q}^{\frac{p}{p_0}}
\nonumber
\\
&\quad +
|z_1-z_2|
\left\|
\left(
\sum_{l=1}^{\infty}
|2^{ls} \varphi_l(D)f|^r
\right)^{\frac{1}{r}}
\right\|_{\cM^p_q}^{\frac{p}{p_0}}
\nonumber
\\
&\lesssim
|z_1-z_2| \|f\|_{\mathcal{E}{}^{s}_{pqr}}^{\frac{p}{p_0}}.
\end{align}
Likewise, 
\begin{align}\label{eq:17421-2}
\|G_1(z_1)- G_1(z_2)\|_{\overset{\diamond}{\mathcal{E}}{}^{s_1}_{p_1q_1r_1}}
\lesssim
|z_1-z_2| \|f\|_{\mathcal{E}{}^{s}_{pqr}}^{\frac{p}{p_1}}.
\end{align}
Therefore, \eqref{eq:17421-1} and \eqref{eq:17421-2} yield 
\begin{align}
\|G(z_1)- G(z_2)\|_{\overset{\diamond}{\mathcal{E}}{}^{s_0}_{p_0q_0r_0}+\overset{\diamond}{\mathcal{E}}{}^{s_1}_{p_1q_1r_1}}
\lesssim
|z_1-z_2|
\left(
\|f\|_{\mathcal{E}{}^{s}_{pqr}}^{\frac{p}{p_0}}
+\|f\|_{\mathcal{E}{}^{s}_{pqr}}^{\frac{p}{p_1}}
\right).
\end{align}
This implies the continuity of $G$.
Furthermore, for every $z\in U$, we have $G'(z)=F(z)$ and
$F(z)\in \overset{\diamond}{\mathcal{E}}{}^{s_0}_{p_0q_0r_0}
+\overset{\diamond}{\mathcal{E}}{}^{s_1}_{p_1q_1r_1}$.
\end{proof}
%
%
Let $k=0,1$ and $t_1,t_2 \in \R$. 
By a similar argument for obtaining \eqref{eq:17423-1}, 
we have 
\begin{equation}\label{eq:160526-2b}
G(k+i t_1)-G(k+i t_2) \in 
{\mathcal E}^{s_k}_{p_kq_kr_k}
\end{equation}
if $f \in {\mathcal E}^{s}_{pqr}$
and
we have 
\begin{equation}\label{eq:160526-2a}
G(k+i t_1)-G(k+i t_2) \in 
\overset{\diamond}{{\mathcal E}}{}^{s_k}_{p_kq_kr_k}
\end{equation}
if $f \in 
\overset{\diamond}{{\mathcal E}}{}^{s}_{pqr}$.

\begin{lemma}\label{lem:17415-4}
Let $k=0,1$. 
Let
$f \in {\mathcal E}^{s}_{pqr}$
with norm $1$.
\begin{enumerate}
\item
Then the function $t\in \R \mapsto G(k+it)-G(k) \in 
{\mathcal E}^{s_k}_{p_kq_kr_k}$ 
is Lipschitz continuous.
\item
Assume $r_0=r_1=r$ and $s_0=s_1=s$.
Let
$f \in 
\overset{\diamond}{{\mathcal E}}{}^{s}_{pqr}$.
Then the function $t\in \R \mapsto G(k+it)-G(k) \in 
\overset{\diamond}{{\mathcal E}}{}^{s_k}_{p_kq_kr_k}$ 
is Lipschitz continuous.
\end{enumerate}
\end{lemma}

\begin{proof}
We suppose
 $f \in  \overset{\diamond}{{\mathcal E}}{}^{s}_{pqr}$.
The case of
$f \in  
{\mathcal E}^{s}_{pqr}$
can be handled similarly.
Let $t_1, t_2 \in \R$ and let $J\in \mathbb{N} \cap [5,\infty)$.
By a similar argument for obtaining \eqref{eq:17423-1}, we have 
\begin{align}\label{eq:17425-1}
\left\|
\sum_{j=J}^{\infty}
\varphi_j(D)\left(
G(it_1) -G(it_2)
\right)
\right\|_{\mathcal{E}^{s_0}_{p_0q_0r_0}}
\lesssim
\|I_1\|_{\cM^{p_0}_{q_0}}
+
\|I_2\|_{\cM^{p_0}_{q_0}}
\end{align}
where 
\begin{align*}
I_1:=
\chi_{[a,a^{-1}]}(S(f;r,s)) 
\left(
\sum_{j=J-1}^\infty
|2^{js} \varphi_j(D)f|^{r}
V_j(f)^{\frac{pr}{p_0}-r}
\left|
\int_{t_2}^{t_1}
V_j(f)^{Qit} \ dt
\right|^{r}
\right)^{\frac{1}{r}}
\end{align*}
and 
\begin{align*}
I_2
&:=
(1-\chi_{[a,a^{-1}]}(S(f;r,s)))
\\
&\times 
\left(
\sum_{j=J-1}^\infty
|2^{js} \varphi_j(D)f|^{r}
V_j(f)^{\frac{pr}{p_0}-r}
\left|
\frac{V_j(f)^{Qit_1}-V_j(f)^{Qit_2}}{\log(V_j(f))}
\right|^{r}
\right)^{\frac{1}{r}}.
\end{align*}
By virtue of Lemma \ref{lem:151109-1}, we get
\begin{align}\label{eq:170403-22}
\|I_1\|_{\cM^{p_0}_{q_0}}
&\lesssim
|t_1-t_2|
\left\|
\chi_{[a,a^{-1}]}(S(f;r,s)) \left(
\sum_{j=J-1}^\infty
|2^{js} \varphi_j(D)f|^{r}
V_j(f)^{\frac{pr}{p_0}-r}
\right)^{\frac{1}{r}}
\right\|_{\cM^{p_0}_{q_0}}
\nonumber
\\
&\lesssim
|t_1-t_2|
\left\|
\chi_{[a,a^{-1}]}(S(f;r,s))
\left(
\sum_{\ell=J-1}^{\infty}
|2^{\ell s} \varphi_\ell(D)f|^{r}
\right)^{\frac{p}{rp_0}}
\right\|_{\cM^{p_0}_{q_0}}
\nonumber
\\
&\le 
|t_1-t_2|
\|S(f;a,J-1,r,s)\|_{\cM^p_q}^{\frac{p}{p_0}}.
\end{align}
We combine  Lemmas \ref{lem:161101-1},
\ref{lem:151112-1},  and \ref{lem:161109-2} to obtain
\begin{align*}
I_2 &\lesssim
(1-\chi_{[a,a^{-1}]}(S(f;r,s)) )
\left(
\sum_{j=0}^{\infty}
|2^{js} \varphi_j(D)f|^{r}
\Phi_{\frac{p}{p_0}}(V_j(f))^{r}
\right)^{\frac{1}{r}}
\\
&\lesssim
(1-\chi_{[a,a^{-1}]}(S(f;r,s)) )
\Psi_{\frac{p}{p_0}}
\left(
S(f;r,s)^r
\right)^{\frac{1}{r}}
\\
&\lesssim
\left(
a^{\frac{p}{p_0}(r-1)}
+
\left(
\log\left(
\sqrt{a}+\frac{1}{\sqrt{a}}
\right)
\right)^{-r}
\right)
S(f;r,s)^{\frac{p}{p_0}}.
\end{align*}
Consequently,
\begin{align}\label{eq:170403-33}
\|I_2\|_{\cM^{p_0}_{q_0}}
\lesssim
\left(
a^{\frac{p}{p_0}(r-1)}
+
\left(
\log\left(
\sqrt{a}+\frac{1}{\sqrt{a}}
\right)
\right)^{-r}
\right)
\|f\|_{\mathcal{E}^s_{pqr}}^{\frac{p}{p_0}}
\end{align}
Therefore, by combining \eqref{eq:17425-1}, \eqref{eq:170403-22} and \eqref{eq:170403-33},  and then taking $J\to \infty$ and $a\to 0^+$, 
we have
\[
\lim\limits_{J\to \infty}
\left\|
\sum_{\ell=J}^{\infty}
\varphi_\ell(D)(G(it_1)-G(it_2))
\right\|_{\mathcal{E}^{s_0}_{p_0q_0r_0}}
=0.
\]
Thus, $G(it_1)-G(it_2) \in \overset{\diamond}{\mathcal{E}}{}^{s_0}_{p_0q_0r_0}$.
The proof of 
$G(1+it_1)-G(1+it_2) \in \overset{\diamond}{\mathcal{E}}{}^{s_1}_{p_1q_1r_1}$ is similar.

Now we prove the second part of this lemma.
From \eqref{eq:161018-1} and Theorem \ref{thm:vec-maxi},  it follows that
\begin{align}\label{eq:170421-4}
\|G(it_1)-G(it_2)\|_{\overset{\diamond}{\mathcal{E}}{}^{s_0}_{p_0q_0r_0}}
&=
\|G(it_1)- G(it_2)\|_{{\mathcal{E}}{}^{s_0}_{p_0q_0r_0}}
\nonumber
\\
&\lesssim
\left\|
\varphi_0(D)f \cdot V_0(f)^{\frac{p}{p_0}-1}
\int_{t_2}^{t_1} V_0(f)^{Qit} \ dt \right\|_{\cM^{p_0}_{q_0}}
\nonumber
\\
&\quad +
\left\|
\varphi_1(D)f \cdot V_1(f)^{\frac{p}{p_0}-1}
\int_{t_2}^{t_1} V_1(f)^{Qit}\ dt 
\right\|_{\cM^{p_0}_{q_0}}
\nonumber
\\
&\quad +
\left\|
\left(
\sum_{l=1}^{\infty}
\left|2^{ls} 
\varphi_l(D)f \cdot V_l(f)^{\frac{p}{p_0}-1}
\int_{t_2}^{t_1} V_l(f)^{Qit} \ dt \right|^{r_0}
\right)^{\frac{1}{r_0}}
\right\|_{\cM^{p_0}_{q_0}}.
\end{align}
By virtue of Lemma \ref{lem:151109-1},
we get
\begin{align*}
\|G(it_1)-G_0(it_2)\|_{\overset{\diamond}{\mathcal{E}}{}^{s_0}_{p_0q_0r_0}} 
&\lesssim
|t_1-t_2|
\|\varphi(D)f\|_{\cM^p_q}^{\frac{p}{p_0}}
+
|t_1-t_2|
\|\varphi_1(D)f\|_{\cM^p_q}^{\frac{p}{p_0}}
\nonumber
\\
&\quad +
|t_1-t_2|
\left\|
\left(
\sum_{l=1}^{\infty}
|2^{ls} \varphi_l(D)f|^r
\right)^{\frac{1}{r}}
\right\|_{\cM^p_q}^{\frac{p}{p_0}}
\lesssim
|t_1-t_2| 
\|f\|_{\mathcal{E}^s_{pqr}}^{\frac{p}{p_0}}.
\end{align*}
By a similar argument, we also have
\[
\|G(1+it_1)-G_0(1+it_2)\|_{\overset{\diamond}{\mathcal{E}}{}^{s_1}_{p_1q_1r_1}} 
\lesssim
|t_1-t_2| 
\|f\|_{\mathcal{E}^s_{pqr}}^{\frac{p}{p_1}},
\]
as desired.
\end{proof}

\section*{Acknowledgement}

The authors are thankful to 
Professors Wen Yuan and Dachun Yang for
his helpful discussion.

\end{document}